\RequirePackage{fix-cm}
\documentclass[]{article}
\usepackage{arxiv}
\usepackage{lipsum}
\usepackage{placeins}
\usepackage[italicComments=true,indLines=true,noEnd=false]{algpseudocodex}
\usepackage{float}
\usepackage[english]{babel}
\usepackage[utf8]{inputenc} 
\usepackage[T1]{fontenc}    
\usepackage{hyperref}       
\usepackage{url}            
\usepackage{booktabs}       
\usepackage{stackrel,amssymb}
\usepackage{amsfonts,amsmath,amstext,amsbsy,amsthm}
\usepackage{nicefrac}       
\usepackage{microtype}      
\usepackage{graphicx}
\usepackage[numbers,sort&compress]{natbib}
\usepackage{doi}

\newtheorem{theorem}{Theorem}[section]

\usepackage{xcolor}
\DeclareMathOperator*{\argmin}{argmin} 
\usepackage{caption}
\usepackage{subcaption}
\usepackage[title]{appendix}
\usepackage[capitalise,nameinlink]{cleveref}

\title{Data-driven Control of Agent-based Models: an Equation/Variable-free Machine Learning Approach}

\author{ {Dimitrios G. Patsatzis} \\
    Dept. of Science and Technology for \\
	Energy and Sustainable Mobility \\
	Consiglio Nazionale delle Ricerche \\
	Naples, Italy \\ 
	 \texttt{dimitrios.patsatzis@stems.cnr.it} \\
	\And
	{Lucia Russo} \\
	Dept. of Science and Technology for \\
	Energy and Sustainable Mobility \\
	Consiglio Nazionale delle Ricerche \\
	Naples, Italy \\ 
    \texttt{lucia.russo@stems.cnr.it}  \\
	\And
	{Ioannis G. Kevrekidis} \\
	Dept. of Chemical and Biomolecular Engineering \& \\Dept. of Applied Mathematics and Statistics \&\\
	Dept. of Medicine, Johns Hopkins University\\
	Baltimore, USA\\
	\texttt{yannisk@jhu.edu} \\
	\And
	{Constantinos Siettos} \thanks{Corresponding author} \\
	Dept. of Mathematics and Applications \&\\
    Scuola Superiore Meridionale\\
    Universit\`a degli Studi di Napoli Federico II\\
	Naples, Italy\\
	\texttt{constantinos.siettos@unina.it} \\
}

\renewcommand{\shorttitle}{Equation/Variable Free Machine Learning}

\hypersetup{
pdftitle={A template for the arxiv style},
pdfsubject={q-bio.NC, q-bio.QM},
pdfauthor={David S.~Hippocampus, Elias D.~Striatum},
pdfkeywords={First keyword, Second keyword, More},
}

\begin{document}
\maketitle

\begin{abstract}
We present an Equation/Variable free machine learning (\emph{EVFML}) framework for the control of the collective dynamics of complex/multiscale systems modelled via microscopic/agent-based simulators. The approach obviates the need for construction of surrogate, reduced-order models.~The proposed implementation consists of three steps: (A) from high-dimensional agent-based simulations, machine learning (in particular, non-linear manifold learning (Diffusion Maps (DMs)) helps identify a set of coarse-grained variables that parametrize the low-dimensional manifold on which the emergent/collective dynamics evolve. The out-of-sample extension and pre-image problems, i.e. the construction of non-linear mappings from the high-dimensional input space to the low-dimensional manifold and back, are solved by coupling DMs with the Nystr\"{o}m extension and Geometric Harmonics, respectively; (B) having identified the manifold and its coordinates, we exploit the Equation-free approach to perform numerical bifurcation analysis of the emergent dynamics; then (C) based on the previous steps, we design data-driven embedded wash-out controllers that drive the agent-based simulators to their intrinsic, imprecisely known, emergent open-loop unstable steady-states, thus demonstrating that the scheme is robust against numerical approximation errors and modelling uncertainty. The efficiency of the framework is illustrated by controlling emergent unstable (i) traveling waves of a deterministic agent-based model of traffic dynamics, and (ii) equilibria of a stochastic financial market agent model with mimesis.
\end{abstract}


\section{Introduction}

Scientific computation and control of the emergent/collective dynamics of high-dimensional multiscale/complex dynamical systems constitute open challenging tasks due to (a) the lack of physical insight and knowledge of the appropriate macroscopic quantities needed to usefully describe the evolution of the emergent dynamics, (b) the so-called ``curse of dimensionality'' when trying to efficiently learn surrogate models with good generalization properties, and (c) the problem of  bridging the scale where individual units (atoms, molecules, cells, bacteria, individuals, robots) interact, and the macroscopic scale where the emergent properties arise and evolve \citep{kevrekidis2003equation,liu2016control,diBernardo2020,karniadakis2021physics}.
For the task of identification of macroscopic variables from high-fidelity simulations/spatio-temporal data, various machine learning methods have been proposed including non-linear manifold learning algorithms such as Diffusion Maps (DMs) \citep{coifman2005geometric,nadler2006diffusion,coifman2008diffusion,singer2009detecting,chiavazzo2014reduced,liu2015equation,koronaki2020data,lee2020coarse,galaris2022numerical}, ISOMAP \citep{balasubramanian2002isomap,bollt2007attractor,bhattacharjee2016nonlinear} and Local Linear Embedding \citep{roweis2000nonlinear,papaioannou2021time} but also Autoencoders \citep{chen2018molecular,vlachas2022multiscale}. 
For the task of the extraction of surrogate models for the approximation of the emergent dynamics, available approaches include the Sparse Identification of the Nonlinear Dynamics (SINDy) \cite{brunton2016discovering}, the Koopman operator \citep{koopman1932dynamical,mezic2013analysis,williams2015data,brunton2016koopman,dietrich2020koopman,mauroy2020koopman}, Gaussian Processes \cite{lee2020coarse,papaioannou2021time,raissi2018hidden}, Artificial Neural Networks (ANNs) \cite{lee2020coarse,galaris2022numerical}, Recursive Neural Networks (RNN) \cite{vlachas2022multiscale}, Deep Learning \citep{raissi2019physics}, as well as Long Short-Term Memory (LSTM) networks \cite{vlachas2018data}.
For the task of the bridging of the micro- and macro-scales, the Equation-free (EF) multiscale framework, introduced back in the early 2000s years, \citep{kevrekidis2003equation,theodoropoulos2000coarse,kevrekidis2004equation}
bypasses the need to extract explicit, closed form macroscopic/surrogate models of any type (e.g. black or gray box ODEs or PDEs) by bridging  the micro- and macro- scales ``on demand''.~If one has {\em a priori} physical insight/knowledge of a good set of macroscopic observables, this bridging  is achieved with the aid of \emph{coarse-timestepping}, which is a discrete mapping in time of the macroscopic observables.~The key assumption behind the EF approach is that the microscopic/fine dynamics are quickly attracted towards a slow, invariant manifold, parametrized by the first few low-order moments of the microscopic/fine distribution \citep{kevrekidis2003equation,zagaris2009analysis,siettos2022numerical}. The EF framework allows the execution of system-level tasks (numerical bifurcation analysis, control, optimization) directly at the macroscopic level using the microscopic/agent-based simulator as a computational experiment, short bursts of whose dynamics can be initialized, run and observed on demand.~Recently, Maclean et al. \cite{maclean2021toolbox} developed a toolbox of EF functions in Matlab.~Based on the concept of EF and pseudo-arc-length continuation for the construction of bifurcation diagrams with feedback control \citep{siettos2004coarse}, several control-based approaches have been also developed for performing experimentally assisted bifurcation analysis (see \citep{sieber2011control,barton2013systematic,renson2019numerical,panagiotopoulos2022continuation} and references therein).~In these works, the experiment plays the role of the ``black box'' microscopic/agent-based/complex system.
What is different here, with respect to computer simulations, is that in ``experimental experiments'' one can not usually initialize the experiment at will - something that can be done relatively easily in ``computational experiments'' \citep{kevrekidis2003equation}.\par

In summary, all the above works focus mainly on either (a) the identification, from high-fidelity simulations, of a set of coarse-grained variables, and then using machine learning (e.g., ANNs, GPR) the construction of reduced-order surrogate black-box models for the emergent dynamics (that are then used for low-fidelity simulations, bifurcation analysis and control), or (b) on the EF coarse-grained numerical analysis and control of the microscopic/agent-based dynamics, assuming that the coarse-grained macroscopic variables are known before-hand. Very recently, Chin et al.\cite{chin2022enabling} used DMs to find a set of macroscopic variables from spatio-temporal data produced through the so-called optimal velocity traffic model (OVM) on a ring-shaped road and, based on these, performed numerical bifurcation analysis of the emergent dynamics within the EF framework.\par
In this work, we propose a fully data-driven Equation/Variable-free framework based on machine learning (\emph{EVFML}) for the control of the emergent dynamics of complex/multiscale dynamical systems modelled via detailed agent-based/microscopic simulators. The approach is deployed in three main steps (see also Fig.(\ref{fig.overalchematic}).
\begin{figure}[ht]
    \centering
    \includegraphics[scale=0.62]{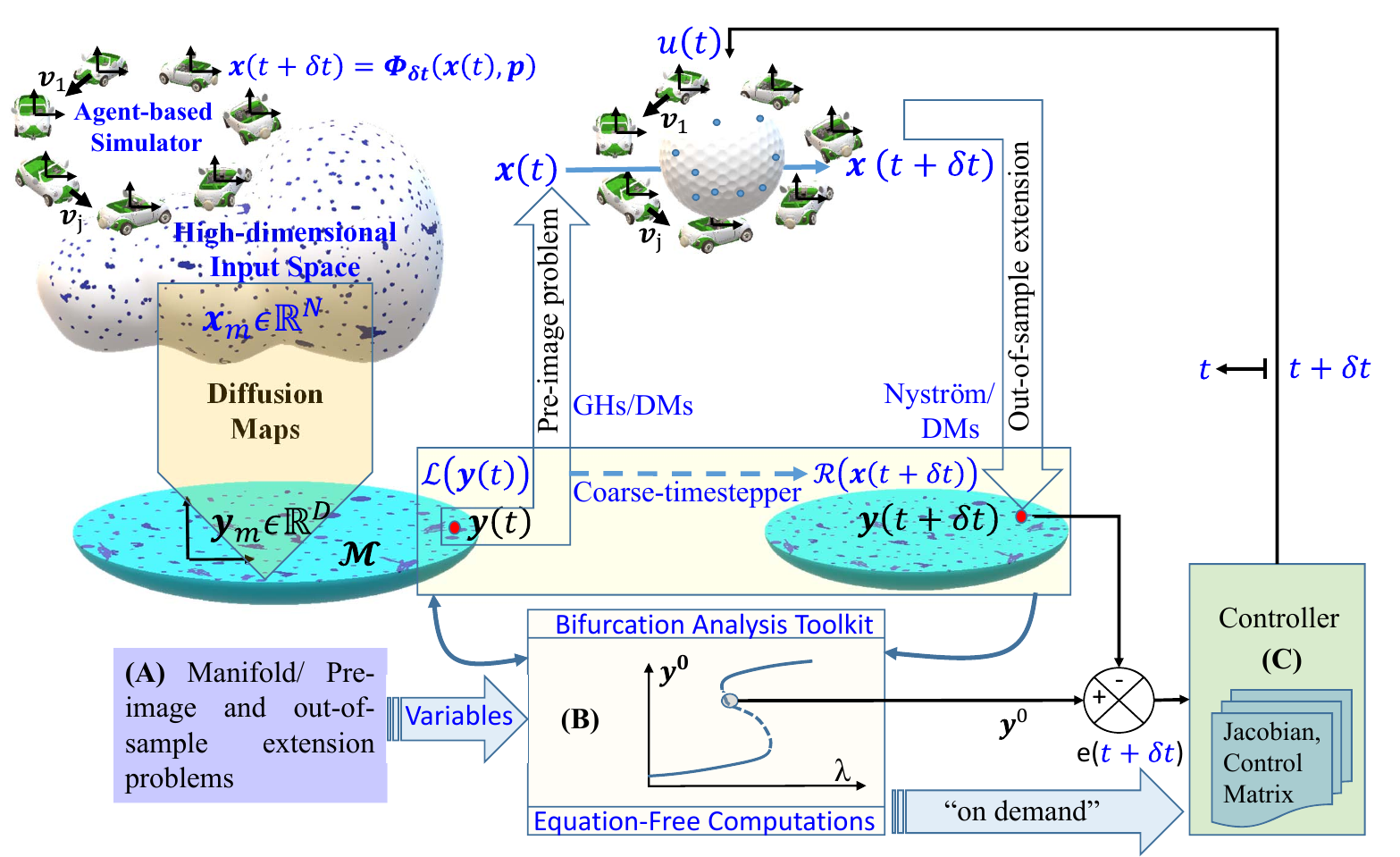} 
    \caption{Schematic of the proposed data-driven \emph{EVFML} framework, based on machine learning, for the control of the emergent/collective dynamics of complex systems modelled via agent-based simulators. The approach is deployed in three main steps: (A) Discovery of the low-dimensional manifold $\mathcal{M}$ and its coordinates using non-linear manifold learning and in particular Diffusion Maps (DMs), from a high-dimensional point cloud, (B) Equation-free numerical bifurcation analysis on $\mathcal{M}$, (C) Repetition of (A) around the coarse-grained steady-states of interest over the control parameter space and ''on demand'' design of wash-out controllers for the emergent dynamics.} 
    \label{fig.overalchematic}
    \end{figure}
In the first step (A), based on detailed agent-based simulations over the bifurcation parameter space, we use machine learning, and in particular non-linear manifold learning (DMs), to discover a set of coarse-grained variables that can be used to define a low-dimensional manifold on which the emergent dynamics evolve. This first step involves also the solution (i) of the out-of-sample extension problem (i.e. the construction of a map that \emph{embeds/restricts/projects} new unseen points in the high-dimensional input space to the low-dimensional manifold) coupling the Nystr\"{o}m method with DMs, and, (ii) of the pre-image problem (i.e. the construction of the non-linear inverse map that \emph{reconstructs/lifts/pre-images} new unseen points in the manifold to the high-dimensional input space), coupling Geometric Harmonics (GHs) with DMs. In the second step (B), based on the identified coarse-grained variables and the construction of the restriction and lifting operators from step A, we exploit the EF approach to perform numerical bifurcation analysis of the emergent dynamics, thus tracing steady-state solution branches of the emergent/collective dynamics and quantifying their stability. In the third step (C), with the information extracted from step (B), we design embedded wash-out controllers (\citep{abed1994stabilization,siettos2012equation}) that drive the agent-based simulators to converge to their intrinsic, imprecisely known, emergent open-loop unstable steady-states. In order to increase the numerical approximation accuracy of the numerical quantities (such as Jacobians and control matrices) required for the design of the controllers, we repeat the procedure of step A around the steady-states of interest over the control parameter space. Thus, we demonstrate that the proposed control scheme is robust against numerical approximation errors/modelling uncertainties introduced in the estimation of the actual values of the coarse-grained steady-states and the numerical solution of the out-of-sample extension and pre-image problems.\par
The performance of the \emph{EVFML} framework is assessed through two problems: (a) the stabilization of unstable traveling waves of a deterministic agent-based simulator describing traffic dynamics along a ring-shaped road \citep{chin2022enabling,marschler2014coarse}, and, (b) the stabilization of unstable equilibria of a stochastic agent-based model describing the dynamics of a simple financial market with mimesis \citep{siettos2012equation,omurtag2006modeling}.\par
We demonstrate that our \emph{EVFML} framework allows the accurate learning and control of the emergent dynamics without any prior knowledge/physical insight  of the ``correct'' macroscopic observables, and despite numerical approximation errors/ modelling uncertainty, thus bypassing the need to construct any type of global surrogate models that {\em de facto} introduce biases. Furthermore, the \emph{EVFML} framework suggests a new way to engineer complex systems via machine/deep learning, as potentially it can be used to facilitate: (a) the data collection in an intelligent/targeted way ``on demand''  around the states of interest over the parameter space, and (b) the bridging of the notional gap between high-dimensional data point clouds, physics interpretability, numerical analysis and control.
%
%
%
%
%
\section{Materials and Methods}
\label{sec:Met}

In what follows, we present the various elements of the three-step proposed Equation/Variable-free machine learning (\emph{EVFML}) framework. We begin with the description of the procedure for the identification of a set of coarse-grained variables that parametrize a low-dimensional manifold on which the emergent dynamics evolve using Diffusion Maps (DMs).~We also present the methodology for solving the out-of-sample extension and pre-image problems coupling DMs with the Nystr\"{o}m method and Geometric Harmonics (GHs), respectively; it should be noted that the solution of the pre-image problem, i.e. that of the inversion of the non-linear feature map, is in general an ill-posed problem and as such its solution is not unique. Next, we provide a brief description of the Equation-free (EF) framework and the concept of the \emph{coarse-timestepper}, which enables the targeted extraction of macroscopic system-level information directly from the detailed/fine scale, microscopic/agent-based simulations.~Finally, we illustrate the ``on-demand'' design of embedded controllers, in particular wash-out filters, to stabilize embedded regular unstable steady-states, thus demonstrating that the proposed scheme is, under certain assumptions, robust to modelling uncertainties introduced in the approximation of the manifold and the construction of the non-linear mappings from the input space to the manifold and back.~For an overview of the various steps, please refer to Fig.~\ref{fig.overalchematic}.
%
\subsection{Step A. Discovery of Coarse-grained variables via Diffusion Maps and the numerical solution of the out-of-sample extension and pre-image problems}
\label{sub:DM}

For discovering coarse-grained variables that parametrize the underlying manifold on which the emergent dynamics evolve, we use non-linear manifold learning and in particular Diffusion Maps (DMs) \citep{coifman2005geometric,coifman2008diffusion,nadler2006diffusion,coifman2006geometric}.~The goal is to construct a non-linear mapping from a high-dimensional space to a low-dimensional subspace, so that the intrinsic geometry of the embedded manifold is preserved.~Given the collected $M$ observations $\mathbf{x}_m\in \mathbb{R}^N, m=1,\ldots,M$ in the high-dimensional space, which can be compactly written in a row-wise matrix form $\mathbf{X}\in\mathbb{R}^{M\times N}$, we assume that these data are embedded in a smooth low-dimensional manifold $\mathcal{M} \subset \mathbb{R}^N$.~DMs seek to find their low-dimensional representations $\mathbf{y}_m\in \mathbb{R}^D$ with $D\ll N$, compactly written in the matrix form $\mathbf{Y}\in\mathbb{R}^{M\times D}$, such that the Euclidean distance of the points $\mathbf{y}$ (rows of $\mathbf{Y}$) is preserved as the diffusion distance of points in $\mathbf{X}$ (rows of $\mathbf{X}$) \citep{nadler2006diffusion}.

The implementation begins by defining a \emph{similarity metric} between pairs of points, say $\mathbf{x}_i$, $\mathbf{x}_j \in\mathbf{X}$, $\forall i,j=1,\ldots M$ in the high-dimensional space.~By utilizing the Euclidean norm $d_{ij}=||\mathbf{x}_i - \mathbf{x}_j||$, a Gaussian kernel function  $k(\mathbf{x}_i,\mathbf{x}_j)$ is employed for calculating the affinity matrix:
\begin{equation}
    \mathbf{A} = \left[a_{ij} \right]= \left[ k(\mathbf{x}_i,\mathbf{x}_j)  \right] = exp \left( -\dfrac{d_{ij}^2}{\epsilon^2}\right) = exp \left( -\dfrac{||\mathbf{x}_i - \mathbf{x}_j||^2}{\epsilon^2}\right),
    \label{eq:AffM}
\end{equation}
where the shape parameter $\epsilon$ expresses a measure of the local neighborhood in the high-dimensional space.\par
Then, one constructs the $M \times M$ Markovian transition matrix $\mathbf{M}$, by normalizing each row of the affinity matrix, such that
\begin{equation}
\mathbf{M}=\mathbf{D}^{-1}\mathbf{A}, \quad \text{where} \quad \mathbf{D}=\text{diag} \left( \sum_{j=1}^{M}a_{ij} \right).
\label{eq:MarkovM}
\end{equation}
The elements $\mu_{ij}$ of $\mathbf{M}$ correspond to the probability of jumping from one point to another in the high-dimensional space.~In particular, the transition matrix defines a Markovian random walk $X_t$ on the data points, with probability of moving from point $i$ to point $j$ at the $t$-step of a conceptual diffusion process across the data: 
$$ \mu_{ij} = \mu(\mathbf{x}_i,\mathbf{x}_j) = \text{Prob}(X_{t+1} = \mathbf{x}_j|X_t = \mathbf{x}_i)$$ 
Given the weighted graph defined via the Gaussian kernel function $k(\mathbf{x}_i,\mathbf{x}_j)$ on the high-dimensional space, the random walk can be then defined by the transition probabilities
$$ \mu_{ij} = \mu(\mathbf{x}_i,\mathbf{x}_j) = \dfrac{k(\mathbf{x}_i,\mathbf{x}_j)}{deg(\mathbf{x}_i)}, \quad \text{where} \quad deg(\mathbf{x}_i) = \sum_{j=1}^M k(\mathbf{x}_i,\mathbf{x}_j),$$
thus retrieving the expression in Eq.~\eqref{eq:MarkovM}.

The Markovian transition matrix $\mathbf{M}$ is similar to the matrix $\mathbf{\hat{M}} =\mathbf{D}^{-1/2}\mathbf{A}\mathbf{D}^{-1/2}$, which is symmetric and positive definite, thus allowing for the eigendecomposition:
\begin{equation}
\mathbf{M}=\sum_{i=1}^M \lambda_i \mathbf{z}_i \mathbf{y}_{i}^{T}, 
\end{equation}
where $\lambda_i \in \mathbb{R}$ are the eigenvalues and $\mathbf{z}_i \in \mathbb{R}^M$ and $\mathbf{y}_i \in \mathbb{R}^M$ are the left and right eigenvectors of $\mathbf{M}$, respectively, such that $\langle \mathbf{z}_i,\mathbf{y}_j \rangle=\delta_i^j $.~The set of right eigenvectors $\mathbf{y}_i$ provides an orthonormal basis for the low-dimensional subspace in $\mathbb{R}^D$ spanned by the rows of $\mathbf{M}$.~The best $D$-dimensional low-rank approximation of the row space of $\mathbf{M}$ in the Euclidean space $\mathbb{R}^D$ is given by the $D$ right eigenvectors corresponding to the $D$ largest eigenvalues.\par
~Thus, the DMs embedding is realized as the mapping of each observation $\mathbf{x}_m$ to the row vector $$\mathbf{y}_m=(y_{1,m}, \ldots, y_{D,m}), \forall m=1,\ldots,M,$$ where $y_{i,m}$ denotes the $m$-th element of the $i$-th right eigenvector, corresponding to the $i$-th non-trivial, sorted in descending order, eigenvalue $\lambda_i$ for $i=1,\ldots,D\ll N$.
~It has been shown that such a selection of DMs eigenvectors provides the best approximation of the Euclidean distance $\rVert \mathbf{y}_i - \mathbf{y}_j \lVert$ between two points, say $\mathbf{y}_i$, $\mathbf{y}_j\in \mathbf{Y}$ on the low-dimensional space to the diffusion distance on the high-dimensional space  \citep{nadler2006diffusion}, defined as:
$$ D_{t}^2(\mathbf{x}_i,\mathbf{x}_j) = \lVert \mu_t(\mathbf{x}_i,\cdot),\mu_t(\mathbf{x}_j,\cdot)\rVert_{L_2,1/deg}^2 = \sum_{k=1}^{M} \dfrac{(\mu_t(\mathbf{x}_i,\mathbf{x}_k)-\mu_t(\mathbf{x}_j,\mathbf{x}_k))^2} {deg(\mathbf{x}_k)} $$
where $\mu_t(\mathbf{x}_i,\cdot)$ is the $i$-th row of the Markovian transition matrix $\mathbf{M}^t$, corresponding to the transition probabilities after $t$ diffusion steps; here, for our computations, we considered $t=1$.\par
In practice, the embedded dimension $D$ is determined by the spectral gap of the eigenvalue ratio of the transition matrix $\mathbf{M}$, assuming that the first $D$ leading eigenvalues are adequate for providing a good approximation of the diffusion distance between all pairs of points \citep{coifman2008diffusion}.~However, this is not always the case and one should consider employing \textit{parsimonious} DMs \citep{dsilva2018parsimonious,holiday2019manifold} for selecting the eigenvectors that
provide the best low-dimensional embedding.\par
For the cases considered in this work, though, such a parsimonious identification was not required, since the spectral gap was an adequate criterion for discovering the DMs eigenvectors that provide the desired manifold parametrization.

\subsection{The numerical solution of the Out-of-sample extension and Pre-image problems}
\label{sub:OoSE}

The coarse-scale variables that parametrize the low-dimensional manifold are provided in a data-driven manner by the DMs embeddings over the parameter space, thus circumventing intuition-based identification.~However, for employing the \emph{coarse-timestepper} of the EF framework, it is required to construct the restriction and lifting operators so that they are able to cope with out-of sample data.~This task requires the numerical solution of the so-called out-of-sample extension and pre-image problems, respectively.
%
%
%
\subsubsection{The Out-of-sample extension problem: the Restriction operator}
\label{subsub:Restr}

Let the high-dimensional data set be $\mathbf{X}=\{\mathbf{x}_m \in \mathbb{R}^N | m=1,\ldots,M\}$ of $M$ observations and assume that the employment of DMs in Section~\ref{sub:DM} results in a low-dimensional embedding with, say, $\lambda_i$ for $i=1,\ldots,D$ sorted eigenvalues and their associated right eigenvectors $\mathbf{y}_i$.~Here, the restriction operator $\mathcal{R}$ evaluated at a point of the given data set $\mathbf{x}_m$ is defined as:
\begin{equation}
    \mathcal{R}(\mathbf{x}_m) = (y_{1,m}, \ldots, y_{D,m}) = \mathbf{y}_m \in \mathbb{R}^D, \qquad m=1,\ldots,M,
    \label{eq:Rest1}
\end{equation}
where $y_{i,m}$ denotes the $m$-th component of the $i$-th DMs eigenvector $\mathbf{y}_i$.~However, in the case where new unseen high-dimensional data points are presented, one needs to recalculate the DMs embeddings.~To avoid the computational cost of such a procedure, we instead utilize the so-called Nystr\"{o}m method \citep{nystrom1929praktische} to extend $\mathcal{R}$ to new unseen data points \citep{coifman2008diffusion}; i.e. solve the out-of-sample extension problem.\par
The extension is based on the fact that the eigenvectors $\mathbf{y}_i$ form a basis of the low-dimensional subspace (see \cite{coifman2006geometric} for details).~This implies that any function defined at known data points $\mathbf{x}_m$ of the high-dimensional space, say, $f(\mathbf{x}_m)$ can be extended to the low-dimensional manifold as:
\begin{equation}
    \hat{f}(\mathbf{x}_m) = \sum_{i=1}^D a_i \mathbf{y}_i(\mathbf{x}_m) \qquad m=1,\ldots,M,
    \label{eq:GHbasis}
\end{equation}
where $a_i=\langle \mathbf{y}_i, f(\mathbf{x}_m)\rangle$ are the projection coefficients of the function to be extended to the first $D$ eigenvectors.~Using the same projection coefficients, one can now map the function evaluated to new unseen data points, say $\mathbf{x}^n_l \in \mathbb{R}^N$ for $l=1,\ldots,L$, of the high-dimensional space as
\begin{equation}
    \boldsymbol{\mathcal{E}}(\hat{f}(\mathbf{x}^n_l)) = \sum_{i=1}^D a_i \mathbf{\hat{y}}_i(\mathbf{x}^n_l),
    \end{equation}
where,
\begin{equation}
    \mathbf{\hat{y}}_i(\mathbf{x}^n_l)=\dfrac{1}{\lambda_i}\sum_{m=1}^M k(\mathbf{x}_m,\mathbf{x}^n_l) \mathbf{y}_{i}(\mathbf{x}_m), \quad i=1,\ldots,D
    \label{eq:GH}
\end{equation}
are the corresponding Geometric Harmonics and $k(\mathbf{x}_m,\mathbf{x}^n_l)$ is the Gaussian kernel function used to seed the construction of the affinity matrix in Eq.~\eqref{eq:AffM}, this time measuring the similarity of th the new point $\mathbf{x}^n_l$ to the known data points $\mathbf{x}_m$.~Casting the data set of new points $\mathbf{x}^n_l$ for $l=1,\ldots,L$ in the matrix form $\mathbf{X}^n \in \mathbb{R}^{L \times N}$, Eq.~\eqref{eq:GH} can be written in the compact matrix form \cite{coifman2006geometric}:
\begin{equation}
   \boldsymbol{\mathcal{E}}(\boldsymbol{\hat{f}})=\mathbf{K}_{L\times M}\mathbf{Y}_{M\times D}{\boldsymbol{\Lambda}_{D\times D}^{-1}}\mathbf{Y}_{D\times M}^T\boldsymbol{f}_{M\times 1},
   \label{eq:GH2}
\end{equation}
where $\mathbf{K}_{L\times M}$ is the corresponding kernel matrix, $\mathbf{Y}_{M\times D}$ is the matrix of the DMs eigenvectors $\mathbf{y}_i$ in columns and $\boldsymbol{\Lambda}_{D\times D}$ is the diagonal matrix with elements $\lambda_i$.\par
For the solution of the out-of-sample extension problem, we seek to extend the restriction operator of Eq.~\eqref{eq:Rest1} for new out-of-sample unseen data points $\mathbf{X}^n \not\subset \mathbf{X}$ of the high-dimensional space.~Thus, substitution of $\boldsymbol{\hat{f}}$ with $\mathcal{R}$ and of $\boldsymbol{f}$ with $\mathbf{Y}$ to Eq.~\eqref{eq:GH2} (the left and right hand sides in Eq.~\eqref{eq:Rest1}) implies
\begin{equation}
   \mathcal{R}(\mathbf{X}^n)=\boldsymbol{K}_{L\times M}\mathbf{Y}_{M\times D}{\boldsymbol{\Lambda}_{D\times D}^{-1}}.
   \label{eq:Rest2}
\end{equation}
The above extends the definition of the restriction operator to the $L$ new data points in $\mathbf{X}^n$ on the manifold as constructed by the DMs embeddings.~A detailed schematic representation of the Nystr\"{o}m method is displayed in Fig.~\ref{fig:Ny-GH_kNN}, where the image $\mathcal{R}(\mathbf{x}^n_l)$ of the new point $\mathbf{x}^n_l$ is depicted with blue dots and arrows.
\begin{figure}[!hb]
    \includegraphics[width=\textwidth]{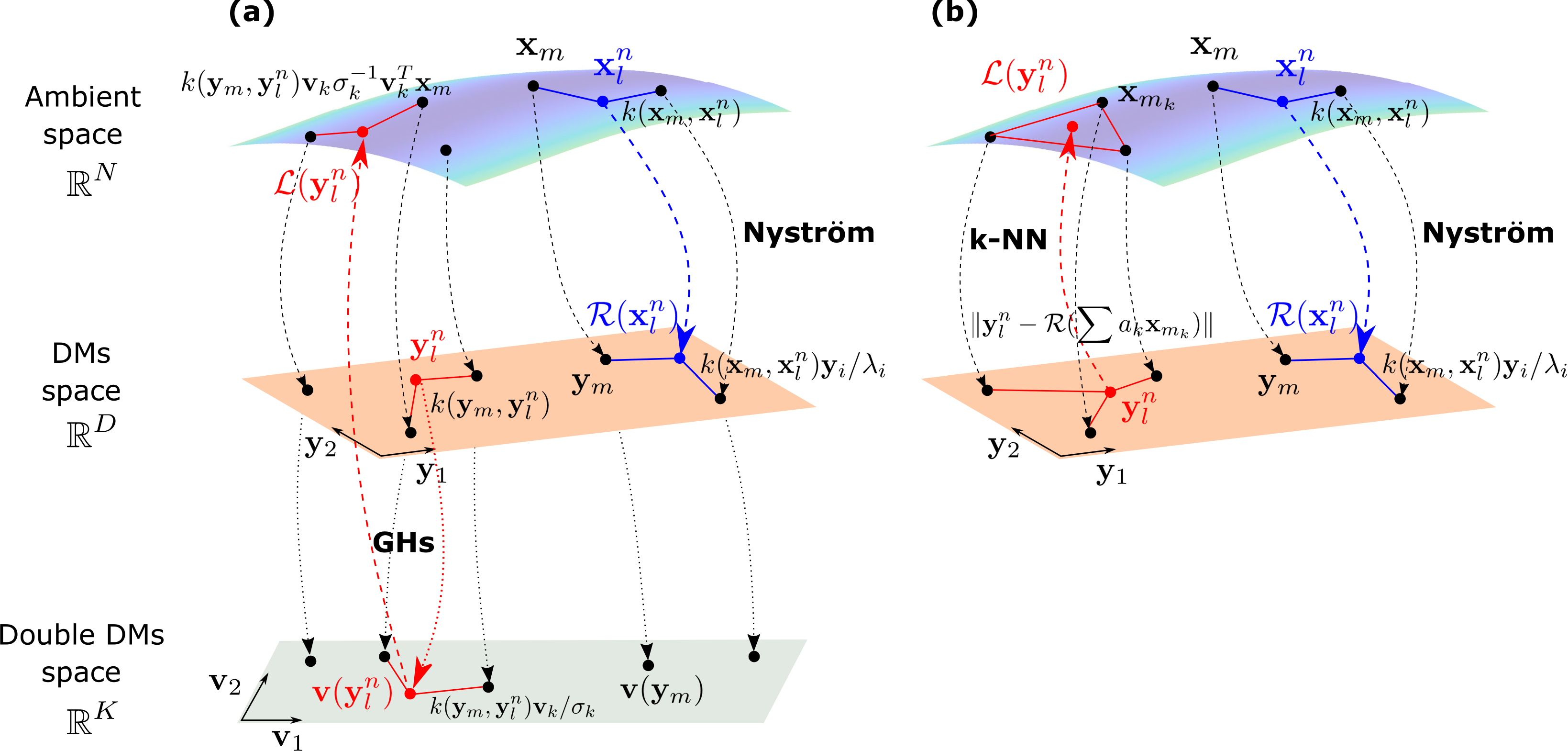} 
    \caption{Schematic representation of the out-of-sample extension problem (blue dots and arrows) via Nystr\"{o}m method and the pre-image problem via GHs (red dots and arrows) at panel (a) and via k-NN at panel (b).~The black dots denote observations in the ambient space $\mathbb{R}^N$ and their representations on the DMs and ``double" DMs spaces $\mathbb{R}^D$ and $\mathbb{R}^K$, respectively.~The Nystr\"{o}m method is performed for obtaining the image $\mathcal{R}(\mathbf{x}^n_l)$ of the new point $\mathbf{x}^n_l$.~The reconstruction with GHs (top panel) and with k-NN (bottom panel) is performed for obtaining the pre-image $\mathcal{L}(\mathbf{y}^n_l)$ of a new point $\mathbf{y}^n_l$.~As described in Section~\ref{subsub:Lift} for the GHs extension, projection to the ``double'' DMs space is first performed for defining a new basis in $\mathbb{R}^K$ through which one obtains the reconstructed state in $\mathbb{R}^N$.}
    \label{fig:Ny-GH_kNN}
\end{figure}
%
%
%
\subsubsection{The Pre-image problem: the Lifting operator.}
\label{subsub:Lift}
The numerical solution of the pre-image problem refers to the reconstruction of high-dimensional representations, given low-dimensional measurements on the manifold.~In terms of the EF analysis, this task corresponds to the construction of a ``lifting'' operator:
\begin{equation}
\mathcal{L}\equiv \mathcal{R}^{-1}: \mathcal{R}(\mathbf{x}) \mapsto \mathbf{x},
\label{eq:Lift1}
\end{equation}
so that new, unseen points in the low-dimensional manifold $\mathbf{y}^n_l \notin \{\mathbf{y}_m = \mathcal{R}(\mathbf{x}_m)| m=1,\ldots,M\}$ for $l=1,\ldots,L$ are mapped back to the original high-dimensional space.~Below, we present the two methodologies utilized in this work to solve the pre-image problem, namely the Geometric Harmonics (GHs) and the k-Nearest Neighbors (k-NN) algorithms.~For a detailed review and comparison of various methods see \citep{chiavazzo2014reduced,evangelou2022double,papaioannou2021time}.
%
%
%
\paragraph{\textbf{Geometric Harmonics (GHs)}} GHs are sets of functions that allow the extension of a function defined on a high-dimensional space to new, unseen points on the low-dimensional manifold \citep{coifman2006geometric}.~Their derivation is based on the Nystr\"{o}m method 
and they constitute the functions shown in  Eq.~\eqref{eq:GH}.\par
Thus, one would employ Eq.~\eqref{eq:GH} for the new points in the manifold $\mathbf{y}^n_l$ for $l=1,\ldots,L$, with the Gaussian kernel function now considered on the low-dimensional space as $k(\mathbf{y}_m,\mathbf{y}^n_l)$, where $\mathbf{y}_m \in \mathcal{R}(\mathbf{X})$ is the DMs embedding of the $m$-th observation.~However, as pointed out by \citep{chiavazzo2014reduced,evangelou2022double}, the basis provided by the DMs embeddings $(\lambda_i, \mathbf{y}_i)$ for $i=1,\ldots,D$ is unable to provide an accurate approximation of the function on the manifold, that is the function of Eq.~\eqref{eq:GHbasis}, which in this case takes the form $\mathbf{\hat{f}}(\mathbf{y}_m)$.~Instead, one should again employ DMs, this time on the $D$ first $\mathbf{y}_i$ eigenvectors, in order to obtain a new basis upon which to extend $\mathbf{\hat{f}}$ \citep{chiavazzo2014reduced,evangelou2022double}.~In what follows, we present the implementation of the ``Double DMs'' technique for calculating the GHs functions.\par
Recall that the employment of the DMs to the high-dimensional data $\mathbf{X}=\{\mathbf{x}_m \in \mathbb{R}^N | m=1,\ldots,M\}$ of $M$ observations over the parameter space result to the embedding $(\lambda_i, \mathbf{y}_i)$ for $i=1,\ldots,D$.~At the first step, similarly to the application of the DMs, a Gaussian kernel function is utilized for constructing the affinity matrix for the $\mathbf{y}_i$ embeddings
\begin{equation}
    \mathbf{\tilde{A}} = \left[\tilde{a}_{ij} \right] = \left[ k(\mathbf{y}_i,\mathbf{y}_j)  \right] = exp \left( -\dfrac{||\mathbf{y}_i - \mathbf{y}_j||^2}{\tilde{\epsilon}^2}\right),
    \label{eq:AffM_GH}
\end{equation}
where $\tilde{\epsilon} \ll \epsilon$ of the ``first'' DMs embedding.~Here, one does not need to calculate a Markovian transition matrix as in the ``first'' DMs, since the accurate approximation of the diffusion distance is no longer needed.~Since $\mathbf{\tilde{A}}$ is a positive and semi-definite matrix, it has a set of non-negative eigenvalues, say $\sigma_k \in \mathbb{R}$, and the corresponding orthonormal right eigenvectors $\mathbf{v}_k \in \mathbb{R}^M$.~The selection of the first $K$ largest eigenvalues, such that $\sigma_K>\delta \sigma_0$, where $\delta>0$ serves as a threshold, determines the accuracy of the new embedding; the smaller/largest $\delta$ is, the more/less accurate the ``second'' DMs embedding is.~We then use the eigenvectors $\mathbf{v}_k$ as a basis to project the function $\mathbf{\hat{f}}$ via the new low-dimensional points $\mathbf{y}^n_l \in \mathbb{R}^D$ for $l=1,\ldots,L$, so that Eq.~\eqref{eq:GH} now reads
\begin{equation*}
    \boldsymbol{\mathcal{E}}(\hat{f}(\mathbf{y}^n_l)) = \sum_{k=0}^K \tilde{a}_k \mathbf{\hat{v}}_k(\mathbf{y}^n_l) \quad \text{where} \quad
    \mathbf{\hat{v}}_k(\mathbf{y}^n_l)=\dfrac{1}{\sigma_k}\sum_{m=1}^M k(\mathbf{y}_m,\mathbf{y}^n_l) \mathbf{v}_{k}(\mathbf{y}_m), \quad k=0,\ldots,K
\end{equation*}
are the GHs, evaluated at the basis obtained via the ``second'' DMs embedding $\mathbf{v}_{k}(\mathbf{y}_m)$.~This time, the projection coefficients of the function $\mathbf{f}(\mathbf{y}_m)$ to be extended are calculated as $\tilde{a}_k = \langle \mathbf{v}_k, \mathbf{f}(\mathbf{y}_m)\rangle$, where $\mathbf{y}_m $ is the DMs embedding of the $m$-th observation.~Casting the extension in the matrix form:
\begin{equation}
   \boldsymbol{\mathcal{E}}(\boldsymbol{\hat{f}})=\mathbf{K}_{L\times M}\mathbf{V}_{M\times K}{\boldsymbol{\Sigma}_{K\times K}^{-1}}\mathbf{V}_{K\times M}^T\boldsymbol{f}_{M\times 1}
   \label{eq:GH_LD2}
\end{equation}
one obtains a similar to Eq.~\eqref{eq:GH2} form, where now $\mathbf{V}_{M\times K}$ is the column-wise matrix of the $K$ eigenvectors $\mathbf{v}_k$ obtained through the ``second'' DMs and $\boldsymbol{\Sigma}_{K\times K}$ is the diagonal matrix with the corresponding eigenvalues $\sigma_k$.\par
Regarding the numerical solution of the pre-image problem, we are interested in constructing the lifting operator for a set of, say $L$, new out-of-sample data points $\mathbf{y}^n_l \notin \{\mathbf{y}_m = \mathcal{R}(\mathbf{x}_m)| m=1,\ldots,M\}$ of the low-dimensional space, compactly written in the matrix form $\mathbf{Y}^n\in\mathbb{R}^{L\times N}$.~Thus, substitution of $\boldsymbol{\hat{f}}$ with $\mathcal{L}$ and of $\boldsymbol{f}$ with $\mathbf{x}$ (the left and right-hand sides of Eq.~\eqref{eq:Lift1}) to Eq.~\eqref{eq:GH_LD2} implies the matrix form
\begin{equation}
   \mathcal{L}(\mathbf{Y}^n)=\mathbf{K}_{L\times M}\mathbf{V}_{M\times K}{\boldsymbol{\Sigma}_{K\times K}^{-1}} \mathbf{V}_{K\times M}^T\mathbf{X}_{M\times N},
   \label{eq:LiftGH}
\end{equation}
which extends the definition of the lifting operator to reconstruct (pre-image) the $L$ new data points in $\mathbf{Y}^n$ in the high-dimensional space.~A detailed schematic representation of the GHs extension is demonstrated in Fig.~\ref{fig:Ny-GH_kNN}(a).\par
It is important here to note that the number $K$ of the ``second'' DMs eigenvectors needs to be high enough for an accurate reconstruction of the high-dimensional space (see \citep{evangelou2022double} for an indicative algorithm on the selection of $\delta$ threshold).~However, the advantage of this method, as evidenced by Eq.~\eqref{eq:LiftGH} is that for the reconstruction of the new data points $\mathbf{Y}^n$, only the matrix $\mathbf{K}$ is required.~This is because the matrix $\mathbf{V}{\boldsymbol{\Sigma}^{-1}} \mathbf{V}^T\mathbf{X}$ is independent of the new data points and thus can be pre-computed for accelerating computations, as also noted by \citep{chiavazzo2014reduced,evangelou2022double}.
%
%
%
\paragraph{\textbf{k-Nearest Neighbors}} The aim of the k-NN approach is to define the lifting operator for an unseen point in the manifold by expressing it as a linear combination of its $k$ nearest neighbors, whose lifted realizations are already known.~This approach involves the numerical solution of an optimization problem \cite{chin2022enabling}.
In particular, one begins by choosing a number $K$ of nearest neighbors, with $K\ge D+1$, near which one assumes that the desired lifted realizations $\mathcal{L}(\mathbf{y}^n_l)$ reside.~The goal is to find such extensions on the high-dimensional space that, when restricted, they will be mapped back as close as possible to the given unseen points in the manifold; i.e.,  $\mathcal{R}(\mathcal{L}(\mathbf{y}^n_l)) \approx \mathbf{y}^n_l$.~We thus assume that the unknown $L$ lifted realizations can be expressed in matrix form as
\begin{equation}
    \mathcal{L} (\mathbf{Y}^n) = \sum_{k=1}^K \mathbf{a}_k \mathbf{x}_{m_k},
\end{equation}
where $\mathbf{x}_{m_k}\in \mathbf{X}$ are data points in the high-dimensional space, the restricted images of which are known via the DMs embeddings; i.e., $\mathcal{R} (\mathbf{x}_{m_k}) = \mathbf{y}_{m_k}$.~Then, one estimates the sets of coefficients $\mathbf{a}_1, \ldots, \mathbf{a}_K \in \mathbb{R}^L$ for the $l$-th point in the manifold through the solution of the optimization problem:
\begin{equation}
    \left( a_{l,1}, \ldots, a_{l,K} \right) = \argmin_{(a_1, \ldots, a_K)\in [0,1]^K} \Bigg\{ \Bigg\| \mathbf{y}^n_l - \mathcal{R}\left( \sum_{i=1}^K a_k \mathbf{x}_{m_k} \right)\Bigg\| : \text{subject to} \sum_{k=1}^K a_k =1 \Bigg\}
    \label{eq:LiftkNN}
\end{equation}
for $l=1, \ldots, L$ where $\mathbf{y}^n_l\in\mathbf{Y}^n$ denotes the $l$-th point in the manifold for which the reconstructed realization is desired.~A schematic representation of the k-NN reconstruction is demonstrated in Fig.~\ref{fig:Ny-GH_kNN}(b).~Note that the optimization problem of Eq.~\eqref{eq:LiftkNN} may in general have multiple solutions, however in practice as the number $K$ of NN increases the minimum gets closer to zero, so that a unique solution is achieved \citep{chin2022enabling}.\par

\subsubsection{Numerical approximation accuracy of the Restriction and Lifting operators.}
\label{subsub:LOOCV}

In order to evaluate the numerical approximation accuracy of the restriction $\mathcal{R}$ operator as computed via the Nystr\"{o}m method and the lifting operator $\mathcal{L}$ as computed by coupling DMs with GHs and k-NN, we used the following leave-out-one cross-validation (LOOCV) procedure.

For the restriction operator $\mathcal{R}$, we considered the DMs embeddings in Eq.~\eqref{eq:Rest1} and removed every time one observation $\mathbf{x}_m$ from the high-dimensional space and its corresponding low-dimensional image $\mathbf{y}_m=\mathcal{R}(\mathbf{x}_m)$; we denote the reduced-by one point-data set by $\mathbf{X}-{\{\mathbf{x}_m\}}$ and its corresponding DMs embedding, by $\mathcal{R}_{\mathbf{X}-{\{\mathbf{x}_m\}}}$.
~Using $\mathcal{R}_{\mathbf{X}-{\{\mathbf{x}_m\}}}$, we apply the Nystr\"{o}m method according to Eq.~\eqref{eq:Rest2} considering $\mathbf{x}_m$ as the new unseen point and compare the resulting $\mathcal{R}(\mathbf{x}_m)$ with the ground truth $\mathbf{y}_m$, $\forall m=1,\ldots,M$.~We assessed the approximation error in terms of $L_1$, $L_2$ and $L_{\infty}$ norms and also reported the absolute and relative error with respect to the ground truth DMs embeddings.\par 
For the lifting operator $\mathcal{L}$, we considered the low-dimensional image provided by the DMs embeddings in Eq.~\eqref{eq:Lift1} and removed every time one low-dimensional image $\mathbf{y}_m$ from the DMs embedding and the corresponding observation $\mathbf{x}_m=\mathcal{L}(\mathbf{y}_m)$; we denote the reduced-by one point-DMs embedding by $\mathbf{Y}-{\{\mathbf{y}_m\}}$ and its corresponding pre-image by $\mathcal{L}_{\mathbf{Y}-{\{\mathbf{y}_m\}}}$.~Using $\mathcal{L}_{\mathbf{Y}-{\{\mathbf{y}_m\}}}$, we employed both GHs and k-NN according to Eqs.~\eqref{eq:LiftGH} and \eqref{eq:LiftkNN}, respectively, considering $\mathbf{y}_m$ as the new unseen point.~We then compared the resulting $\mathcal{L}(\mathbf{y}_m)$ with the ground truth data points in the high-dimensional space $\mathbf{x}_m$, $\forall m=1,\ldots,M$.~We again assessed the approximation error in terms of $L_1$, $L_2$ and $L_{\infty}$ norms in the high-dimensional space.~For visualization purposes (plotting in the low-dimensional space), we additionally employed the restriction operator for calculating the absolute and relative error of $\mathcal{R}(\mathcal{L}_{\mathbf{Y}- \mathbf{y}_m})$ with respect to the ground truth DMs embeddings $\mathbf{y}_m=\mathcal{R}(\mathbf{x}_m)$.~For the cases under study, we report the norm errors of both the restriction and lifting operators in the Supplementary Material, where we also include the visualization of the absolute and relative errors.\par
%
%
%
%
\subsection{Step B. Data-driven Equation/Variable-free numerical bifurcation analysis}
\label{sub:EF_DMs}
Given the data-driven identification of the coarse-grained variables via DMs and the construction of the lifting (via GHs or k-NN) and restriction (via the Nystr\"{o}m method) operators, we have put all the pieces together for the construction of the \emph{coarse-timestepper} which is the core of EF multiscale computations defined on the low-dimensional manifold. Based on this, we perform a numerical bifurcation analysis in order to track branches of steady-state-solutions and extract the numerical quantities needed for the design of controllers \citep{siettos2004coarse,armaou2004time,siettos2006equation,siettos2012equation}.
Equation-free computations rely on the fact that many complex/multiscale dynamical systems exhibit a collective/emergent behavior that can be described by just a few coarse-grained variables describing the dynamics at a macroscopic level \citep{kevrekidis2003equation,kevrekidis2009equation}.~Let us consider a high, say $N$-dimensional dynamical system describing the microscopic evolution of $N$ state variables.~We assume that there exists a low, say $D\ll N$-dimensional slow invariant manifold $\mathcal{M}$, which attracts all the higher-order moments of the microscopic distribution.
Within the EF framework, one computes the necessary macroscopic numerical quantities required for the system-level analysis ``on demand'' on $\mathcal{M}$, by short bursts of appropriately initialized high-dimensional simulations of the agent-based/microscopic model \citep{kevrekidis2003equation,liu2015equation,theodoropoulos2000coarse,kevrekidis2004equation,siettos2012equation,kevrekidis2009equation}.~The key elements here are the restriction, $\mathcal{R}$, and lifting, $\mathcal{L}$, operators.
Ideally, the composition $\mathcal{R} \circ \mathcal{L}$ should map to the identity map $\mathbf{I}_{\mathbb{R}^D}$.\par
\begin{figure}[!hb]
    \centering
    \includegraphics[scale=0.5]{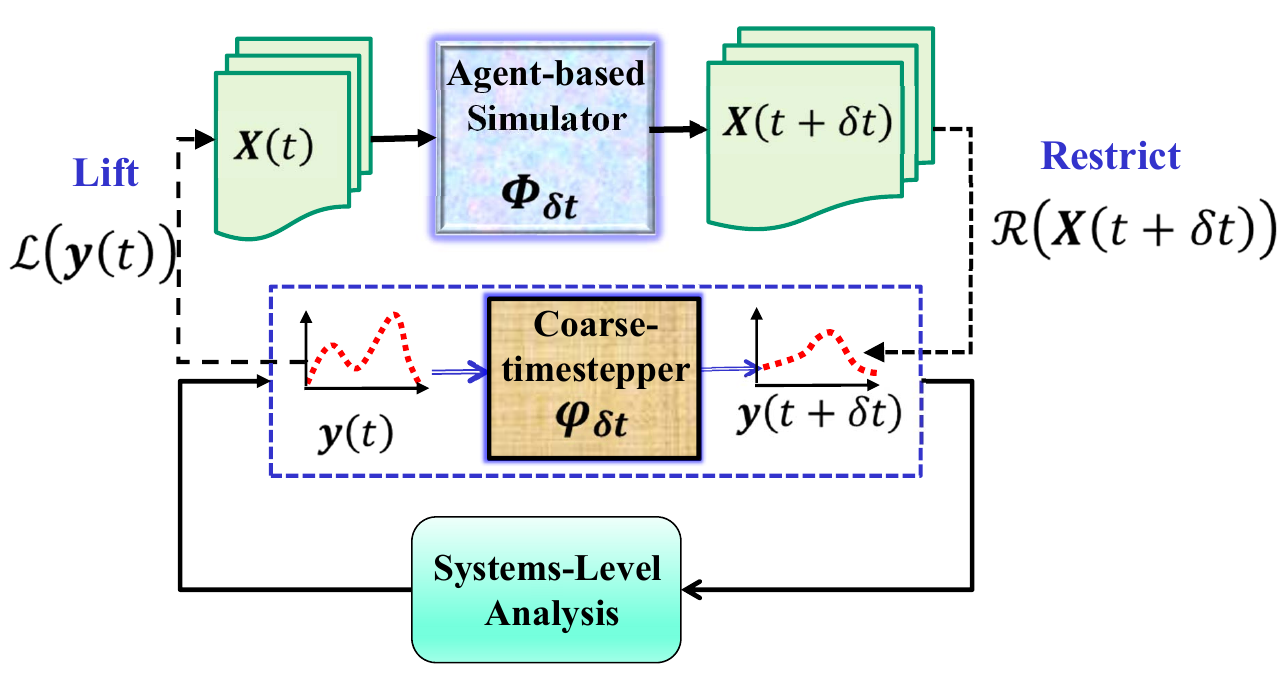} 
    \caption{A schematic of the Equation-free multiscale framework.}
    \label{fig:EFschematic}
\end{figure}
The core of EF analysis is the \emph{coarse-timestepper} \citep{kevrekidis2003equation,kevrekidis2004equation}, which is a ``black-box'' evolution operator of the coarse-scale dynamics on the low-dimensional manifold $\mathcal{M}$.~It essentially constructs a discrete map of the coarse-scale variables  over a macroscopic short time period $\delta t$ by using the lifting and restriction operators to bridge the emergent dynamics with the microscopic scale. The \emph{coarse-timestepper} consists of the following steps (see also Fig.~\ref{fig:EFschematic} for a schematic representation): 
\begin{enumerate} 
    \item Given a macroscopic representation $\mathbf{y}(t) \in \mathbb{R}^D$, map using the lifting operator to a consistent microscopic realization $\mathbf{x}(t) \in \mathbb{R}^N$: $\mathbf{x}(t) = \mathcal{L}(\mathbf{y}(t))$.
    \item Evolve the microscopic realization with the microscopic simulator (agent based, ODE, etc.) for a small macroscopic burst $\delta t$: $$\mathbf{x}(t+\delta t) = \boldsymbol{\Phi}_{\delta t} \left( \mathcal{L}(\mathbf{y}(t)),\boldsymbol{p}(t) \right),$$ where $\boldsymbol{\Phi}_{\delta t}:\mathbb{R}^N\times \mathbb{R}^q\rightarrow \mathbb{R}^N$ denotes the evolution operator of the microscopic simulator, $\boldsymbol{p}\in \mathbb{R}^q$ denotes the vector of parameters including bifurcation parameters, say $\lambda$ and control variables, say $\boldsymbol{u}$.
    \item Map back to the macroscopic level through the restriction operator: $\mathbf{y}(t+\delta t) = \mathcal{R}\left( \mathbf{x}(t+\delta t) \right)$,
\end{enumerate}
thus resulting to
\begin{equation}
 \mathbf{y}(t+\delta t) = \mathcal{R}\left(  \boldsymbol{\Phi}_{\delta t} \left( \mathcal{L}(\mathbf{y}(t)),\boldsymbol{p}(t) \right) \right) \equiv \boldsymbol{\phi}_{\delta t } (\mathbf{y}(t),\boldsymbol{p}(t)),
 \label{eq:CTS}
\end{equation}
where $\boldsymbol{\phi}_{\delta t}:\mathbb{R}^D\times \mathbb{R}^q\rightarrow \mathbb{R}^D$ denotes the unknown explicitly, coarse-grained evolution operator defined on the low-dimensional manifold.~Note that the macroscopic time step $\delta t$ should be long enough in order to allow the microscopic simulations to \textit{heal}; i.e., the microscopic systems' lifted realization to return to the neighborhood of the manifold $\mathcal{M}$.    

%
%
%
In order to allow for the coarse-grained macroscopic system to \textit{heal}, the \emph{coarse-timestepper} in Eq.~\eqref{eq:CTS} is modified as \cite{marschler2014coarse}:
\begin{equation}
    \mathbf{y}(t_{skip}+\delta; \lambda) = \mathcal{R}(\boldsymbol{\Phi}_{\delta}(\mathcal{L}(\mathbf{y}(t_{skip};\lambda)),\boldsymbol{u}^*))=\boldsymbol{\phi}_{\delta}(\mathbf{y}(t_{skip}; \lambda),\boldsymbol{u}^*),
    \label{eq:CTS_imp}
\end{equation}
where $t_{skip}$ is the additional healing step that allows for transients to decay and $\delta$ is the time step to record any changes in the potentially examined stationary points.~Being interested in bifurcation analysis, we explicitly denote by $\lambda$ the bifurcation parameter and consider the control variables in $\mathbf{p}(t)$ constant at $\mathbf{u}^*$.\par
Hence, for tracing the coarse-grained stationary points, one needs to track the solution branches of:
\begin{equation}
    \mathbf{F} (\mathbf{y}, \lambda) \equiv \mathbf{y} - \boldsymbol{\phi}_{\delta}(\mathbf{y}, \lambda) =  0,  \label{eq:BifEq}
\end{equation}
in which we drop the dependency on $t_{skip}$ and $\mathbf{u}^*$ and explicitly denote that of the bifurcation parameter $\lambda$.\par
Here, we employ an equation-free pseudo-arc-length continuation and augment Eq.~\eqref{eq:BifEq} with the  condition:
\begin{equation}
N(\mathbf{y}, \lambda) \equiv \dfrac{(\mathbf{y}_1-\mathbf{y}_0)^T}{\delta \lambda} \cdot (\mathbf{y}-\mathbf{y}_1) + \dfrac{(\lambda_1-\lambda_0)}{\delta \lambda} \cdot (\lambda-\lambda_1) - \delta \lambda = 0,
\label{eq:ContEq}
\end{equation}
where $(\mathbf{y}_1, \lambda_1)$ and $(\mathbf{y}_0, \lambda_0)$ are two previous computed solutions and $\delta \lambda$ is the pseudo-arc-length step size.~The $(D+1)$-dim. system of Eqs.~(\ref{eq:BifEq}, \ref{eq:ContEq}) essentially constitutes a predictor-corrector scheme, which is solved by using the Newton's method on the correction step.~The Newton's iterative update of the $k$-th iteration reads:
\begin{equation*}
    (\mathbf{y}^{k+1}, \lambda^{k+1})^T = (\mathbf{y}^k, \lambda^k)^T - \mathbf{J}^{-1} \cdot [\mathbf{F} (\mathbf{y}^k, \lambda^k), N(\mathbf{y}^k, \lambda^k) ]^T,
\end{equation*}
where the Jacobian $\mathbf{J} \in \mathbb{R}^{(D+1) \times (D+1)}$ 
\begin{equation}
    \mathbf{J} = \begin{bmatrix} \mathbf{F}_{\mathbf{y}}(\mathbf{y}, \lambda) & \mathbf{F}_{\lambda}(\mathbf{y}, \lambda) \\ \dfrac{(\mathbf{y}_1-\mathbf{y}_0)^T}{\delta \lambda} & \dfrac{\lambda_1-\lambda_0}{\delta \lambda} \end{bmatrix}_{(\mathbf{y}^k,\lambda^k)}
    \label{eq:J_NR}
\end{equation}
is estimated numerically, e.g. via finite differences with the use of the \emph{coarse-timestepper} in Eq.~\eqref{eq:CTS_imp} as follows:
\begin{equation}
    \mathbf{F}_{\mathbf{y}} (\mathbf{y}, \lambda) = \mathbf{I} - \dfrac{\boldsymbol{\phi}_{\delta}(\mathbf{y}+\Delta \mathbf{y}, \lambda)-\boldsymbol{\phi}_{\delta}(\mathbf{y}, \lambda)}{\Delta \mathbf{y}}, \qquad \mathbf{F}_{\lambda}(\mathbf{y}, \lambda) = - \dfrac{\boldsymbol{\phi}_{\delta}(\mathbf{y}, \lambda+\Delta \lambda)-\boldsymbol{\phi}_{\delta}(\mathbf{y}, \lambda)}{\Delta \lambda}
    \label{eq:EF_FD}
\end{equation}
considering small perturbations $\Delta \mathbf{y}$ and $\Delta \lambda$ in each element of the coarse variable and the bifurcation parameter, respectively.\par
Here, we re-iterate that in the above procedure, all the necessary numerical quantities required for the numerical bifurcation analysis (coarse-grained steady-states, stability, Jacobians) are provided \textit{on} the manifold parametrized by DMs embeddings $\mathbf{y}_i \in \mathbb{R}^D$, though the \emph{coarse-timestepper} with the appropriate lifting and restriction operators, described in Section~\ref{sub:DM}.\par
%
%
%
\subsection{Step C. Design and implementation of embedded wash-out controllers}
Having performed the numerical bifurcation analysis, we can now proceed with the design of controllers for the stabilization of \textit{regular} unstable steady-states. Here, we design linear \textit{dynamic state feedback} wash-out controllers \citep{panagiotopoulos2022continuation,abed1994stabilization,siettos2012equation} 
based on the numerical quantities extracted by the \emph{coarse-timestepper} around the steady-state that we seek to stabilize. We note that in general, the \emph{coarse-timestepper} defined by Eq.(\ref{eq:CTS}) provides a numerical ``on demand'' approximation of the actual evolution operator of the detailed high-dimensional agent-based dynamics on the manifold around the actual embedded steady-states. Assuming that we have correctly identified the actual intrinsic dimension of the manifold and a set of coarse-grained variables that parametrize it, the modelling uncertainty is mainly due to the numerical errors introduced by the implementation of the restriction and lifting operators.~In general, the coarse-grained steady-states of the agent-based simulator computed by the coarse-grained bifurcation analysis with the EF approach are an approximation of the actual ones. Hence, in order to be able to drive the detailed agent-based simulator at its own actual coarse-grained open-loop unstable states, one has to take into account this modelling uncertainly. Thus, to increase accuracy of the numerical approximation  of the ``local'' manifold around the steady-state of interest, we used again DMs and reconstructed the lifting and restriction operators on a denser data set over the control variable space produced through the agent-based simulator.\par
At this point, we demonstrate the following theorem on the robustness of the proposed control scheme subject to modelling uncertainties and numerical approximation errors.
\begin{theorem}
Let $\mathcal{C} \subset \mathbb{R}^D$ be the set of all \textit{regular} points $\boldsymbol{y}_c$ in the manifold $\mathcal{M}$ parametrized by the DMs coordinates, from which the actual open-loop coarse-grained embedded equilibrium, say $\left( \boldsymbol{y}^*,\boldsymbol{u}^*\right)$ is controllable, i.e. starting from $\boldsymbol{y}_c \in\mathcal{C}$ one can design a controller that can bring the system to $\left( \boldsymbol{y}^*,\boldsymbol{u}^*\right)$ at a finite time. Let us denote by $\mathcal{C}_0 \subset \mathcal{C}$ the subset of all points around $\left( \boldsymbol{y}^*,\boldsymbol{u}^*\right)$ and by $(\mathbf{y}^0,\boldsymbol{u}^*)\in \mathcal{C}_0$ the numerical approximation of $\left( \boldsymbol{y}^*,\boldsymbol{u}^*\right)$ as computed through the constructed coarse-timestepper.\par
Let 
\begin{equation}
    \hat{\mathbf{y}}_{k+1}=\mathbf{A} \hat{\mathbf{y}}_k + \mathbf{B} \hat{\boldsymbol{u}}_k
    \label{eq:CSopen}
\end{equation}
be the linearized system around $(\mathbf{y}^0,\boldsymbol{u}^*)$ as computed with the Equation-free approach through the coarse-timestepper; $\hat{\mathbf{y}}_k=\mathbf{y}_k - \mathbf{y}^0 \in \mathcal{C}_0$, $\hat{\boldsymbol{u}}_k = \boldsymbol{u}_k-\boldsymbol{u}^*$, $\mathbf{A} = \left( \partial \boldsymbol{\phi}_{\delta t }(\mathbf{y}_k, \boldsymbol{u}_k) \right)/ \left( \partial \mathbf{y}_k \right)\big|_{(\mathbf{y}^0, \boldsymbol{u}^*)}$ is the $D\times D$ Jacobian matrix of the system as computed around the equilibrium $(\mathbf{y}^0,\boldsymbol{u}^*)$ and $\mathbf{B} = \left( \partial \boldsymbol{\phi}_{\delta t } (\mathbf{y}_k, \boldsymbol{u}_k) \right)/ \left( \partial \boldsymbol{u}_k \right)\big|_{(\mathbf{y}^0, \boldsymbol{u}^*)}$ is the corresponding $D\times q$ control matrix.\\
Let
\begin{equation}
    \hat{\mathbf{y}}'_{k+1}=\overline{\mathbf{A}} \hat{\mathbf{y}}'_k + \overline{\mathbf{B}} \hat{\boldsymbol{u}}_k
    \label{eq:CSnominal}
\end{equation}
be the actual linearized system around $\left( \boldsymbol{y}^*,\boldsymbol{u}^*\right)$;
$\hat{\mathbf{y}}'_k=\mathbf{y}_k - \mathbf{y}^*= \mathbf{\hat{y}}_{k}+\boldsymbol{\epsilon}\in \mathcal{C}_0$, $\hat{\boldsymbol{u}}_k = \boldsymbol{u}_k-\boldsymbol{u}^*$, $\overline{\mathbf{A}}=\mathbf{A}+\boldsymbol{\Delta}_A$, $\overline{\mathbf{B}}=\mathbf{B}+\boldsymbol{\Delta}_B$, with $\boldsymbol{\Delta}_A, \boldsymbol{\Delta}_B$ representing the modelling uncertainty parts with respect to the actual embedded Jacobian matrix and control matrix, respectively.\\
Then, the controller given by:
\begin{align}
    \boldsymbol{u}_k = \boldsymbol{u}^* + \mathbf{K}^T(\mathbf{y}_k - \mathbf{y}^0) + \mathbf{D} \mathbf{w}_k,\label{eq:WOf1}\\
    \mathbf{w}_{k+1} =\mathbf{w}_k +  (\boldsymbol{u}_k -\boldsymbol{u}^*),
    \label{eq:WOf2}
\end{align}
where $\mathbf{w}_k \in \mathbb{R}^q$ 
``washes-out'' the modelling uncertainty and drives any initial condition $\boldsymbol{y}_c\in \mathcal{C}_0$ to the actual embedded equilibrium $\left( \boldsymbol{y}^*,\boldsymbol{u}^*\right)$ under appropriate selection of the values of the gains  $\mathbf{K} \in \mathbb{R}^{D\times q}$, $\mathbf{D} \in \mathbb{R}^{q\times q}$. We also assume that (a) the closed-loop Jacobian matrix, say $\boldsymbol{A}_C$ is diagonalizable, and (b) the modelling uncertainty is such that $||\boldsymbol{\Delta}_A + \boldsymbol{\Delta}_B\mathbf{K}^T||_1<\delta(\boldsymbol{\Delta}_A,\boldsymbol{\Delta}_B)\ll 1$.
\end{theorem}
\begin{proof}
Without loss of generality, we consider the case of $q=1$. Inserting the control law given by Eqs.~(\ref{eq:WOf1},\ref{eq:WOf2}) into Eq.~\eqref{eq:CSopen} one gets the closed-loop system:
\begin{equation}
    \begin{bmatrix} \mathbf{\hat{y}}_{k+1} \\ w_{k+1} \end{bmatrix} = \begin{bmatrix} \mathbf{A} & \mathbf{0} \\ \mathbf{0} & 1  \end{bmatrix} \cdot \begin{bmatrix} \mathbf{\hat{y}}_k \\ w_k \end{bmatrix} + \begin{bmatrix} \mathbf{B} \\ 1 \end{bmatrix} \hat{u}_{k} = \boldsymbol{A}_C \cdot \begin{bmatrix} \mathbf{\hat{y}}_k \\ w_k \end{bmatrix}, \qquad \boldsymbol{A}_C=\begin{bmatrix} \mathbf{A} + \mathbf{B}\mathbf{K}^T & \mathbf{B} D \\ \mathbf{K}^T & 1+D  \end{bmatrix}
    \label{eq:augCS} 
\end{equation}
Assuming that $(\mathbf{y}^*,\boldsymbol{u}^*)$ (and its approximation) is controllable, one can select the gains $\boldsymbol{K},D$ (e.g. using pole placement techniques or/optimal linear control) so that all the moduli of the eigenvalues of the closed-loop system (\ref{eq:augCS}) are within the unit circle.
This implies that the spectral radius of the matrix $\boldsymbol{A}_C$ is less than one, i.e.,
\begin{equation}
\rho(\boldsymbol{A}_C)\equiv\max\{|\lambda_1|, |\lambda_2|, \dots, |\lambda_D|\}=1-\delta(\boldsymbol{A}_C)<1.
\label{eq:spectralradius}
\end{equation}
The diagonalizability of $\boldsymbol{A}_C$ implies that $\exists \boldsymbol{V}$ that is an invertible matrix such that $\boldsymbol{A}_C=\boldsymbol{V}\boldsymbol{\Lambda}\boldsymbol{V}^{-1}$, where $\boldsymbol{\Lambda}$ it the diagonal matrix with the eigenvalues $\lambda_i,i=1,\dots,D$ of $ \boldsymbol{A}_C$. 
Now by applying the control law given by Eqs.~(\ref{eq:WOf1},\ref{eq:WOf2}) into Eq.~\eqref{eq:CSnominal}, we get the \textit{actual} embedded closed-loop system
\begin{align}
    \begin{bmatrix} \mathbf{\hat{y}}'_{k+1} \\ w_{k+1} \end{bmatrix} = \begin{bmatrix} \overline{\mathbf{A}} & \mathbf{0} \\ \mathbf{0} & 1  \end{bmatrix} \cdot \begin{bmatrix} \mathbf{\hat{y}}'_k \\ w_k \end{bmatrix} + \begin{bmatrix} \overline{\mathbf{B}} \\ 1 \end{bmatrix} \hat{u}_{k} = \left ( \boldsymbol{A}_C +\boldsymbol{\Delta}_{A_C}\right )\cdot \begin{bmatrix} \mathbf{\hat{y}}'_k \\ w_k \end{bmatrix}-\begin{bmatrix} (\mathbf{B}+\boldsymbol{\Delta}_{B})\mathbf{K}^T \\ \mathbf{K}^T \end{bmatrix}\boldsymbol{\epsilon}, \label{eq:augCSnominal0} \\
    \boldsymbol{\Delta}_{A_C}=\begin{bmatrix} \boldsymbol{\Delta}_A + \boldsymbol{\Delta}_B\mathbf{K}^T & \boldsymbol{\Delta}_B D \\ \mathbf{0} & 0  \end{bmatrix}.
    \label{eq:augCSnominal} 
\end{align}
According to the Theorem of Bauer-Fike \cite{bauer1960norms}, $\forall \mu$ eigenvalue of the perturbed matrix $\boldsymbol{A}_C +\boldsymbol{\Delta}_{A_C}$, $\exists \lambda_j \in \boldsymbol{\Lambda}$ such that the following inequality is satisfied:
\begin{equation}
    |\mu-\lambda_j|\le ||\mathbf{V}||_p ||\mathbf{V}^{-1}||_p ||\boldsymbol{\Delta}_{A_C}||_p.
    \label{eq:bauer}
\end{equation}
The above inequality defines the distance between the eigenvalues of $\boldsymbol{A}_C$ and those of $\boldsymbol{A}_C +\boldsymbol{\Delta}_{A_C}$.\par
Let us denote by $\rho(\boldsymbol{A}_C+\boldsymbol{\Delta}_{A_C})\equiv\max\{|\mu_1|, |\mu_2|, \dots, |\mu_D|\}$ the spectral radius of the matrix $\boldsymbol{A}_C+\boldsymbol{\Delta}_{A_C}$.\par
A corollary of the Bauer-Fike Theorem \cite{bauer1960norms} theorem in Eq. (\ref{eq:bauer}) is that 
\begin{equation}
   \rho(\boldsymbol{A}_C+\boldsymbol{\Delta}_{A_C}) < \rho(\boldsymbol{A}_C)+||\mathbf{V}||_p ||\mathbf{V}^{-1}||_p ||\boldsymbol{\Delta}_{A_C}||_p,
    \label{eq:bauercol}
\end{equation}
which, by substituting Eq.~(\ref{eq:spectralradius}) implies
\begin{equation}
   \rho(\boldsymbol{A}_C+\boldsymbol{\Delta}_{A_C}) < 1-\delta(\boldsymbol{A}_C)+||\mathbf{V}||_p ||\mathbf{V}^{-1}||_p ||\boldsymbol{\Delta}_{A_C}||_p.
    \label{eq:bauercol2}
\end{equation}
Thus, for satisfying that $\rho(\boldsymbol{A}_C+\boldsymbol{\Delta}_{A_C})<1$, we have that $||\boldsymbol{\Delta}_{A_C}||_p$ has to be bounded as:
\begin{align}
||\boldsymbol{\Delta}_{A_C}||_p<\frac{\delta(\mathbf{A}_C)}{||\mathbf{V}||_p ||\mathbf{V}^{-1}||_p}.
 \label{eq:stabcond}
\end{align}
Now, by the assumption that $\left( \boldsymbol{y}^*,\boldsymbol{u}^*\right)$ is controllable, from Eq.(\ref{eq:augCSnominal0},(\ref{eq:augCSnominal}) at the steady-state, we have that
\begin{align}
w^{*}=w^{*}+\boldsymbol{K}^T\mathbf{\hat{y}}'+D w^{*}-\boldsymbol{K}^T\boldsymbol{\epsilon}, \quad \mbox or \quad
\boldsymbol{K}^T\mathbf{\hat{y}}'+D w^{*}-\boldsymbol{K}^T\boldsymbol{\epsilon}=0, \quad
\mbox or \quad
\boldsymbol{K}^T(\mathbf{\hat{y}}'-\boldsymbol{\epsilon})+D w^{*}=0 \quad \mbox or \nonumber \\ \boldsymbol{K}^T\mathbf{\hat{y}}+D w^{*}=0.
\label{eq:zerow}
\end{align}
Hence, from Eq.~(\ref{eq:WOf1}), Eq.~(\ref{eq:zerow}) implies that at steady-state, we have that $\hat{u}=0$, i.e.  $u=u^*$.
Therefore, as we consider the stabilization of regular embedded equilibria, the closed-loop system converges to the actual coarse-grained open-loop unstable equilibrium, i.e., $\boldsymbol{y}\rightarrow  \boldsymbol{y}^*$; thus the modelling uncertainty is washed-out and ``absorbed'' in the variable $w$, which at steady state takes the value
\begin{equation}
w^*=-\frac{1}{D}\boldsymbol{K}^T(\mathbf{\hat{y}}'-\boldsymbol{\epsilon})=-\frac{1}{D}\boldsymbol{K}^T\mathbf{\boldsymbol{\hat{y}}}=-\frac{1}{D}\boldsymbol{K}^T(\mathbf{\boldsymbol{y}^*-\boldsymbol{y}^0})=\frac{1}{D}\boldsymbol{K}^T\mathbf{\boldsymbol{\epsilon}}. 
\end{equation}
\end{proof}
Here, for the computation of the gain matrices $\mathbf{K}$ and $D$, we employed a discrete linear quadratic regulator (LQR) technique such that the state feedback control law in Eq.~\eqref{eq:WOf1} minimizes the cost function $J=\sum_{k=0}^\infty \left( \mathbf{y}_k^T \mathbf{Q} \mathbf{y}_k + u_k R u_k \right)$ subject to the state dynamics, where $\mathbf{Q}=0.1 \mathbf{I}$ and $R=1$ (as also considered in \citep{siettos2012equation}).~Such a selection of optimal gain matrices $\mathbf{K}$ and $D$ results to a fast convergence to the coarse-grained unstable steady-states.

\section{Case study 1. Control of unstable traveling waves of an agent-based model of traffic dynamics}
\label{sec:1ring}

As a first benchmark example, we consider a deterministic agent-based simulator describing the unidirectional movement of $N$ (autonomous) cars around a one-lane ring-shaped road of length $L$.~This model has been studied in other works, including the EF bifurcation analysis of the emergent dynamics \citep{chin2022enabling,marschler2014implicit}.
~The dynamics are given by the optimal velocity model (OVM) (a simplified $N$ body simulation) given by \citep{marschler2014coarse,bando1995dynamical}:
\begin{equation}
    \tau \Ddot{x}_n + \Dot{x}_n = V(\Delta x_{n}), \qquad n=1,\ldots,N.
    \label{eq:TM_init}
\end{equation}
$x_n$ denotes the position of each car, $\tau$ the inertia of each car and $V(\cdot)$ is the optimal velocity function that models the response of car $n$.~This response depends only on the position of the car in front, i.e., the $n$-th (autonomous) car adjusts its velocity according to the headway $\Delta x_n = x_{n+1}-x_n$.~The car's movement around the ring-shaped road implies the periodic boundary condition:
\begin{equation}
    x_{n+N} = x_n + L.
    \label{eq:BC}
\end{equation}
For this model, an analytical solution exists \citep{marschler2014implicit,gaididei2009analytical}, that serves for assessing the numerical approximation accuracy of the proposed \emph{EVFML} scheme.
~Note that the system of Eq.~\eqref{eq:TM_init} can be written at its first-order form,
\begin{equation}
    \dot{x}_n = y_n, \qquad \qquad \dot{y}_n = \tau^{-1} \left( V( \Delta x_n) - y_n \right),
    \label{eq:TM}
\end{equation}
where $y_n$ is the velocity of the $n$-th car.~We choose the optimal velocity function \citep{bando1995dynamical} as: 
\begin{equation}
    V(\Delta x_n) = v_0 (tanh(\Delta x_n -h) + tanh(h)),
    \label{eq:OVM}
\end{equation}
where $v_0$ is an optimal velocity parameter related to the maximum allowed velocity $v_0(1+tanh(h))$ and $h$ is the desired safety distance between cars.~The choice of optimal velocity function in Eq.~\eqref{eq:OVM} leads to analytic solutions for the steady-states of the system in Eq.~\eqref{eq:TM} \citep{gaididei2009analytical}.\par
The system in Eq.~\eqref{eq:TM} attains two types of emergent behavior \citep{chin2022enabling,marschler2014implicit,gaididei2009analytical}, namely (a) uniform free-flow, and (b) travelling wave solutions.~For the free-flow solutions:
\begin{equation*}
    x_n(t) = (n-1) L/N + t V(L/N), \quad y_n(t) = V(L/N), \quad n=1,\ldots,N,
\end{equation*}
all cars keep the same distance $\Delta x_n = L/N$ and velocity $V(L/N)$.~In the case of travelling waves, the analytic solution is captured by the ansatz:
\begin{equation*}
    x_n(t) = x_*(n-ct), \quad y_n(t) = y_*(n-ct), \quad n=1,\ldots,N,
\end{equation*}
where $c$ is the wave speed and $x_*(\xi),~ y_*(\xi)$ are periodic functions, the analytic expressions of which are defined in \citep{gaididei2009analytical}.~The travelling wave solution gives rise to traffic jams periodically moving around the ring with speed $c$, where the majority of cars have small distances from the cars in front (small headways) and move with low velocities, while the rest of them have large headways and move with high velocities (see the travelling wave profiles in Fig.~\ref{fig:TrM_AIO}(d)).\par
In all agent-based simulations, we considered $N=20$, $\tau^{-1}=1.7$, $L=40$, with the nominal value of the safety distance set at $h=2.4$.~For the numerical solution of the coupled system of $2N$ ODEs, we used the \texttt{ode45} ODE solver of the MATLAB ode suite, with relative and absolute tolerances set to $10^{-8}$.
\subsection{\textbf{Step A}. Discovery of coarse-grained variables via Diffusion Maps and the numerical solution of the out-of-sample extension and pre-image problems}
\label{subsub:DM_TrM_Bif}
Chin et al. \cite{chin2022enabling} constructed the coarse-grained bifurcation diagram of the OVM model with the aid of the EF framework by (a) considering that the coarse-variables (standard deviation of the headways) are known beforehand as also done in \cite{marschler2014implicit}, and (b) by first applying DMs to discover a proper coarse-grained set of variables from the agent-based simulations, thus using the k-NN algorithm for the solution of the pre-image problem.\par
Here, we used DMs to discover a set of appropriate coarse-grained variables over the bifurcation parameter space and GHs for a numerical solution of the pre-image problem.~Thus, we also compared the performance of the proposed GHs scheme with the k-NN scheme in terms of the numerical approximation accuracy against the analytical results. For our computations, as in \cite{chin2022enabling}, we generated $M=5,000$ initial conditions by perturbing the free-flow solution with a sinusoidal function of amplitude $\mu$.
\begin{equation}
    x_n(0) = (n-1) L /N + \mu \sin(2\pi n/N), \quad y_n(0) = V(L/N), \quad n=1,\ldots,N.
    \label{eq:DM_IC}
\end{equation}
The value of $\mu$ was sampled from a uniform distribution in the interval $\mathcal{U}(0,4.5)$ and we numerically integrated the system of $2N$ ODEs until a time that was sampled from the uniform distribution in the interval $ \mathcal{U}(200,700)$, thus keeping just the final point of each trajectory.~Regarding the bifurcation parameter (the optimal velocity parameter $v_0$), we have sampled its values from the uniform distribution in $\mathcal{U}(0.98,1.08)$.~Hence, the data set $\mathbf{X}\in \mathbb{R}^{M \times N}$ that was fed for the DMs analysis comprises $M=5,000$ observations of $N=20$ headways, including snapshots from both free-flow and travelling wave solutions.\par
We computed the first 20 leading DMs by selecting $\epsilon=3.38$ after employing the approach described in \citep{singer2009detecting}, thus retaining a high similarity between observations ($\sim 61\%$ of the pairwise distances $d_{ij}$ is greater than $\epsilon$).~The related eigenvalues $\lambda_i$ are displayed in Fig.~\ref{fig:TrM_AIO}(a) indicating a large spectral gap between $\lambda_2$ and $\lambda_3$.
\begin{figure}[!htbp]
\centering
	\includegraphics[width=0.94\textwidth]{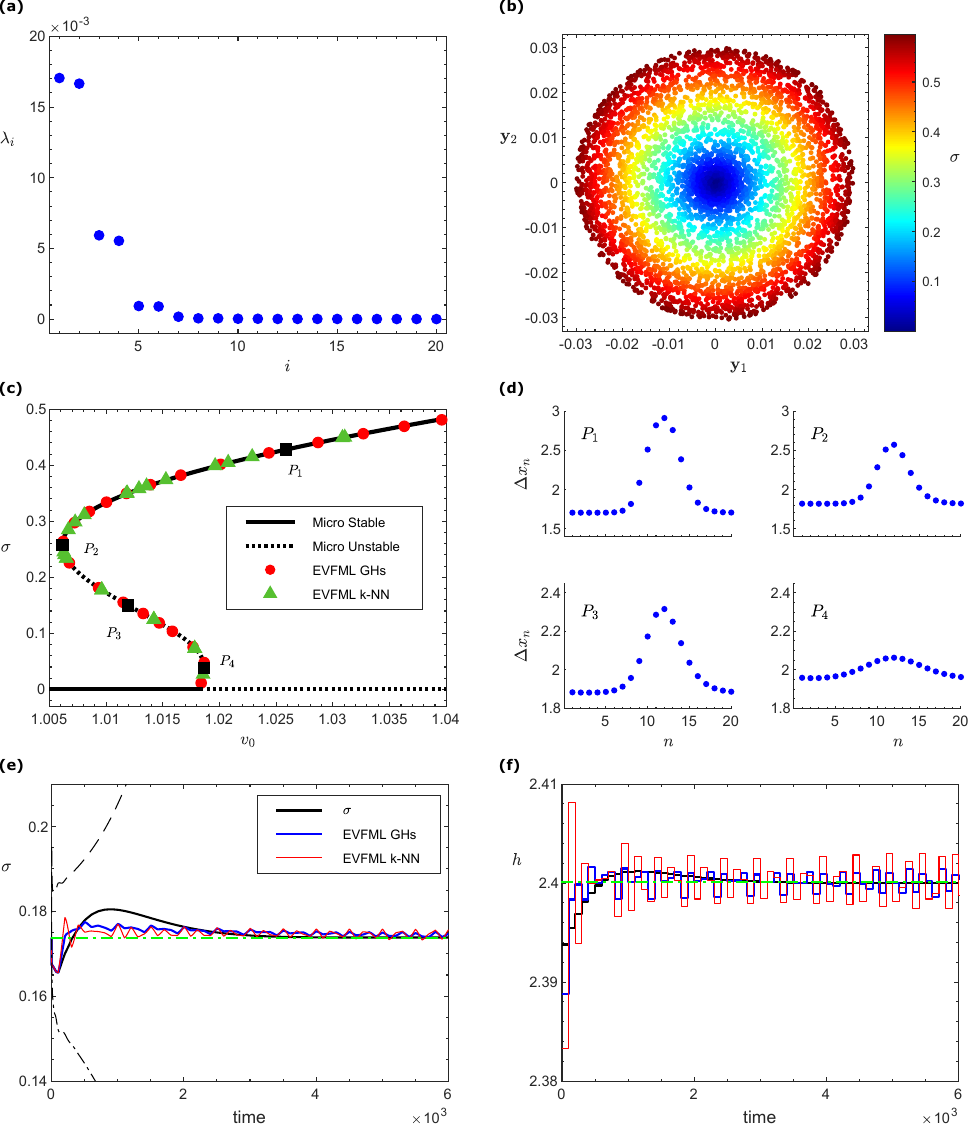}
	\caption{Agent-based traffic model. (a) The first 20 eigenvalues of the DMs coordinates over the bifurcation parameter space, and (b) the corresponding first two DMs coordinates, colored by the standard deviation $\sigma$. (c) Coarse-grained bifurcation diagram constructed through the \emph{EVFML} framework by considering the first 2 leading DMs coordinates as coarse-grained variables; for lifting, both GHs and k-NN were employed.~The analytical derived bifurcation diagram is superimposed; solid (dotted) lines denote stable (unstable) steady-states. The non-zero solution branch corresponds to traveling waves solutions, while the zero solution branch corresponds to free-flow solutions.~(d) Travelling wave profiles at selected points ($P_1$-$P_4$); notice the change in the scale of the y-axis.
	(e),(f) \emph{EVFML} Control of the agent-based traffic dynamics. Implementation of the \emph{EVFML} controller for stabilizing the coarse-grained unstable traveling wave at $v_0 =1.0099$, using $h$ as control variable.~(e) Coarse-grained open-loop and closed-loop responses in $\sigma$ coordinates, and (f) the response of the control variable $h(t)$. The \emph{EVFML} controller drives and stabilizes the emergent dynamics to the actual open loop unstable steady-state. The results obtained assuming the prior knowledge of the coarse-grained variable ($\sigma$) are also given for comparison reasons.}
	\label{fig:TrM_AIO}
\end{figure}
Thus, the first 2 leading DMs $(\mathbf{y}_1, \mathbf{y}_2)$ appear to be good/adequate candidates for parametrizing the manifold.~This assumption is validated with the available prior knowledge from \citep{marschler2014implicit,chin2022enabling}, where it was shown that the standard deviation $\sigma$ of the headways $\Delta x_n$ is an appropriate coarse-grained variable for describing the system's emergent behavior.~Figure~\ref{fig:TrM_AIO}(b) depicts the embedding of the snapshots in the $D=2$-dimensional manifold.~As it is shown: (i) the DMs coordinates $(\mathbf{y}_1, \mathbf{y}_2)$ parametrize the data set through polar coordinates, thus retrieving the underlying geometry of the ring road, and (ii) their radial direction is linearly related to the intuition-based coarse-grained variable $\sigma$.~Hence, it is computationally validated that the $D=2$-dimensional DMs embedding provides an appropriate set of coarse-grained variables to parametrize the underlying manifold.~In particular, only the radial direction is sufficient to describe the system's emergent behavior, in agreement with \citep{chin2022enabling}.\par
The numerical approximation accuracy of the restriction operator $\mathcal{R}$ (as constructed by coupling DMs with the Nystr\"{o}m method) and the lifting $\mathcal{L}$ operator (as constructed by coupling DMs with  GHs) over the bifurcation parameter space is presented in the Supplementary Material S1.1.~Using a leave-out-one cross validation (LOOCV) procedure, as described in Section~\ref{subsub:LOOCV}, it is shown that the proposed GHs scheme provides a higher numerical approximation accuracy when compared to the k-NN algorithm.

\subsection{\textbf{Step B}. Equation/Variable-free numerical bifurcation analysis}
\label{subsub:BifD_TrM}
Based on the restriction and lifting operators constructed in step A, we proceed with the construction of the \emph{coarse-timestepper} (Eq.~\eqref{eq:CTS}) with the coarse-grained variables being the DMs coordinates $(\mathbf{y}_1, \mathbf{y}_2)$.~Since the DMs coordinates lie on the disk shown in Fig.~\ref{fig:TrM_AIO}(b), we performed a polar-coordinate transformation  and considered their one-dimensional radial component as the unique coarse-grained variable.~We then ``wrapped'' around the \emph{coarse-timestepper} the pseudo-arc-length continuation scheme (see Section~\ref{sub:EF_DMs}).~Fig.~\ref{fig:TrM_AIO}(c) shows the coarse-grained bifurcation diagrams as obtained using: (i) the analytical solution, (ii) the \emph{EVFML} framework using DMs and GHs for lifting, and (ii) the \emph{EVFML} approach using DMs and the k-NN algorithm for lifting; the steady-states are displayed in $\sigma$ coordinates to allow comparison with the bifurcation diagram constructed with the analytical solution.~For the numerical simulations, we considered a healing step of $t_{skip}=300$ and an evolution time step (see Eq.~\eqref{eq:CTS_imp}) of $\delta=240$, while the continuation step sizes were set as $\delta v_0=0.005$ in the case of GHs, and $\delta v_0 = 0.0025$ in the k-NN case.
As Fig.~\ref{fig:TrM_AIO}(c) shows, the bifurcation diagrams derived on the basis of the \emph{EVFML} framework are almost identical to the analytical solution.~The free-flow steady-state solutions corresponding to $\sigma=0$ are stable for $v_0 < 1.018$ and become unstable for higher values of $v_0$.~At $v_0=1.018$, a stable travelling wave ($\sigma>0$) solution branch arises, which becomes unstable after the fold point $(v_0,\sigma) \approx (1.0186, 0.0396)$.~Another fold point exists at $(v_0,\sigma) \approx (1.0062, 0.2575)$, thus rendering the traveling wave solution branch stable once again.~After the second fold point, only stable traveling wave solutions exist, with the amplitude of the wave increasing as $v_0$ attains larger values.~This is apparent from the travelling wave profiles shown in Fig.~\ref{fig:TrM_AIO}(d) for indicative points along the traffic jam branch.~Clearly, the bifurcation behavior of the traffic jam branch gives rise to hysteresis phenomena, which have also been detected in \citep{marschler2014implicit,chin2022enabling} who used different sets of values, namely, $N=60,L=60$ and $N=30,L=60$, respectively.~However, therein, no stable travelling wave solution branch for low values of $\sigma$ was detected as in our case that we used a $N=20,L=40$ configuration (see $P_4$ in Fig.~\ref{fig:TrM_AIO}(d) for an indicative solution profile).
%
%
%
\subsection{\textbf{Step C}. Equation/Variable-free control}
\label{sub:EF_cont_TrM}
From the numerical bifurcation analysis, we now have all the  quantities required for the design of wash-out controllers on the embedded space (the Jacobian matrix $\mathbf{A}$ and the control matrix $\mathbf{B}$). For our illustrations, we have considered the stabilization of the emergent behavior at $v_0=1.0099$
(see Fig.~\ref{fig:TrM_AIO}(c)) selecting the safety distance $h$ as our control variable.~Clearly, small perturbations from the unstable steady-state lead the open-loop dynamics either towards to the upper stable solution branch of travelling waves or towards to the lower stable free-flow solution branch (dashed and dashed-dotted trajectories in Fig~\ref{fig:TrM_AIO}(e), respectively).
%
\subsubsection{Reconstruction of the low-dimensional manifold over the control parameter space}
In order to increase the numerical accuracy for the implementation of the embedded controller, we repeated the procedure of step A, this time over the control parameter space around the embedded steady-states that we seek to stabilize. Thus, we fixed $v_0$ at its nominal value and varied the control parameter $h$ around its nominal value to create a denser  data set around the steady-state of interest.~To generate initial conditions, we used  Eq.~\eqref{eq:DM_IC} and considered the amplitude $\mu$ to vary in an interval appropriate for generating trajectories close to the unstable steady-state.~For our computations, the new data set for the design of the embedded controller consisted of 25 trajectories, in $t\in[0, 1600]$ with a sampling time of $t_r = 4$, initialized with an amplitude of $\mu \in [0.2,1.8]$ every $d\mu = 0.4$ and with variation on the control parameter $h\in [2.39, 2.41]$ with $dh=0.005$.~Hence, the resulting new data set $\mathbf{X}\in \mathbb{R}^{M \times N}$ contained $M=10,000$ observations of $N=20$ headways including trajectories attracted to both the free flow and travelling wave stable steady-states.
We computed the first 20 leading DMs, this time by considering $\epsilon=1.16$.~As demonstrated in the Supplementary Material S1.2, the first 2 leading DMs $(\mathbf{y}_1, \mathbf{y}_2)$ appear to be an appropriate set of coarse-grained variables, as (i) there is a large spectral gap between $\lambda_2$ and $\lambda_3$ (see Fig.~S2(a)), and (ii) these 2 DMs coordinates retrieve the intrinsic geometry of the ring-shaped road data set with the radial direction again being linearly related to $\sigma$ (see Fig.~S2(b)).\par
The numerical approximation accuracy of the restriction $\mathcal{R}$ and the lifting $\mathcal{L}$ operators over the control parameter space was assessed using a randomly distributed subset of $5,000$ observations drawn from the new local data set $\mathbf{X}$. A comparison with the results obtained with the k-NN algorithm is also provided.~As demonstrated in the Supplementary Material S1.2, following a LOOCV, the proposed GHs scheme results to a higher numerical approximation accuracy compared to the k-NN algorithm.
\subsubsection{Design and implementation of the embedded wash-out controller}
Based on the constructed \emph{coarse-timestepper}, we designed and implemented a wash-out controller on the basis of the DMs coordinates $(\mathbf{y}_1, \mathbf{y}_2)$.~Given that $(\mathbf{y}_1, \mathbf{y}_2)$ lie on a disk, as shown in Fig.~\ref{fig:TrM_AIO}(b), we transformed the DMs embeddings in polar coordinates and designed the wash-out controller on the basis of the one-dimensional radial component.~The sampling time (reporting horizon) was set to $T=100$.~Based on the \emph{coarse-timestepper}, we computed the Jacobian matrix $\mathbf{A}$ and control matrix $\mathbf{B}$ ``on demand'' and through Linear Quadratic Regulation (LQR), we obtained the optimal gain matrices $\mathbf{K}$ and $D$, which in the case of the lifting operator constructed via GHs are $\mathbf{K}= 20.51$ and $D = 0.243$, while in the case of k-NN are $\mathbf{K} = 19.59$ and $D = 0.249$.~The corresponding emergent closed-loop responses are shown in Fig.~\ref{fig:TrM_AIO}(e,f) in $\sigma$ coordinates.\par
For comparison purposes with the intuition/physics-based coarse-grained variable $\sigma$, we also designed a wash-out controller directly based on $\sigma$.
For this task via LQR, we estimated the gain matrices to be $\mathbf{K}=0.744$ and $D=0.264$.~The corresponding closed-loop response of the system is shown in Fig~\ref{fig:TrM_AIO}(e,f).
As shown in Fig~\ref{fig:TrM_AIO}(e), all controllers successfully drive the system to the actual emergent open-loop unstable travelling wave state at $\sigma=0.1738$, while the open-loop emergent response diverges from the desired steady-state (black dashed and dashed-dotted curves). The \emph{EVFML} control response fluctuates slightly around the nominal value of $h_0=2.4$ due to the corresponding numerical approximation errors introduced by the numerical solution of the out-of-sample extension and pre-image problems.
When compared to the ``$\sigma$-controller'', the \emph{EVFML} controllers result in a smaller initial divergence from the unstable steady-state.~Comparing the two different lifting operators (the one constructed with GHs and the one constructed with k-NN), the proposed GHs scheme results in smaller fluctuations of the control response $h(t)$ around its nominal value; both lifting extensions introduce relative changes less than 0.25\% with respect to the nominal value.
%
%
%
%
\section{Case study 2. Control of unstable stationary states of a stochastic agent-based model of a simple financial market with mimesis}
\label{sec:TradMim}
Our second agent-based model describes the stochastic interactions of a large set of $N$ identical agents trading the same asset \cite{omurtag2006modeling,siettos2012equation}.~The internal state of each agent is described by the propensity to buy or sell an asset, that is considered a real numerical value in $[-1,1]$.~Each agent receives good as well as bad exogenous news about the market, arriving randomly to each agent with Poisson rates $\nu_{ex}^+$ and $\nu_{ex}^-$, respectively.~The response of an agent to good/bad news is the discrete increase/decrease jump of its state by $\epsilon^+$/$\epsilon^-$.~Hence, the internal state of each agent $X_i$ for $i=1,\ldots,N$ changes in an event-driven fashion upon the information Poisson arrival times $t_k$, $k=1,2,\ldots$.~Between these arrival times, when no news arrive, the tendency of an agent to buy or sell fades out, so that its state drifts exponentially to zero with a decay rate constant $\gamma$.~In addition, when an agent's state arrives/surpasses the limits $\pm$1, the agent immediately buys or sells and returns to its neutral state at zero.\par  
The agents follow a mimetic behavior, thus they are influenced from the overall buying or selling trend of the market, say, $R^+$ and $R^-$, respectively, defined as the fractions of agents buying or selling per unit time.~The $R^+$ and $R^-$ change the frequency of arrival of the exogenous, either good or bad, news such that
\begin{equation}
\nu^{\pm} = \nu_{ex}^{\pm}+gR^{\pm},
\label{eq:nuR}
\end{equation}
where $g$ in Eq.~\eqref{eq:nuR} quantifies the average number of social links of the agents.~Since each agent buys/sells in a discrete time manner, we define the overall buying and selling rates as averages over a small, finite interval say $t_{rep}$, 
physically reflecting that the gross activity is not broadcast instantaneously, but is instead reported at regular short time intervals (see \cite{siettos2012equation}).
After the buying/selling action of an agent (i.e., when $X_i> 1$ or $X_i< -1$), the agent's internal state returns to neutral 
and the overall buying/selling rates increase in the next interval $t_{rep}$, thus, the probability of all other agents to buy/sell increases.~Such a formulation bears strong analogies to integrate-and-fire models formulating neuronal cell activity \citep{omurtag2006modeling}

For our simulations, we considered $N=5,000$ agents with $\epsilon^+=0.075$, $\epsilon^-=-0.072$, $\nu_{ex}^+=\nu_{ex}^-=20$ and $\gamma=1$ and the reporting horizon set to $t_{rep}=0.25$.
~This model has been studied in our previous work \citep{siettos2012equation}, where we have performed an EF bifurcation analysis of the emergent dynamics, considering as macroscopic variables the Inverse Cumulative Distribution Function (ICDF) of the microscopic density function and the buying and selling rates.~Here, we assume that no such information is available, thus we employed the \emph{EVFML} framework to discover a correct set of macroscopic observables over the bifurcation parameter space and based on them to design embedded wash-out controllers.
\subsection{\textbf{Step A}. Discovery of coarse-grained variables via Diffusion Maps and the numerical solution of the out-of-sample extension and pre-image problems}
\label{sub:EFbifD_TiM}

Using the stochastic agent-based simulator, we generated $M=5,000$ random initial conditions given by
\begin{equation}
X_i(0) \in \big\{ \mathcal{N}(\mu,0.32^2): |X_i(0)| < 1 \big\}, \qquad R^{\pm}(0)=0, \qquad i=1,\ldots,N,
\label{eq:ICs_TiM}
\end{equation}
such that the initial states $X_i(0)$ are drawn from a Gaussian distribution with fixed standard deviation $\sigma = 0.32$.~The value of $\mu$ was sampled from a uniform distribution in the interval $\mathcal{U}(0,0.45)$ and the bifurcation parameter, $g$, was sampled from a uniform distribution in $\mathcal{U}(38,46)$.~We numerically integrated in time until $t=20 t_{rep}$ and kept the final point and the $t=10 t_{rep}$-th time point of each trajectory.~In the cases where solutions were blowing up (i.e., when $\overline{X}>0.5$) before the final reporting time, we kept the final point of this limit and the one at the middle of this time interval.~Hence, the data set $\mathbf{X}\in \mathbb{R}^{M\times N}$ that was fed for the DMs analysis consisted of $M=5,000$ observations of the $N=5,000$ internal states.\par
\begin{figure}[!htbp]
\centering
	\includegraphics[width=0.94\textwidth]{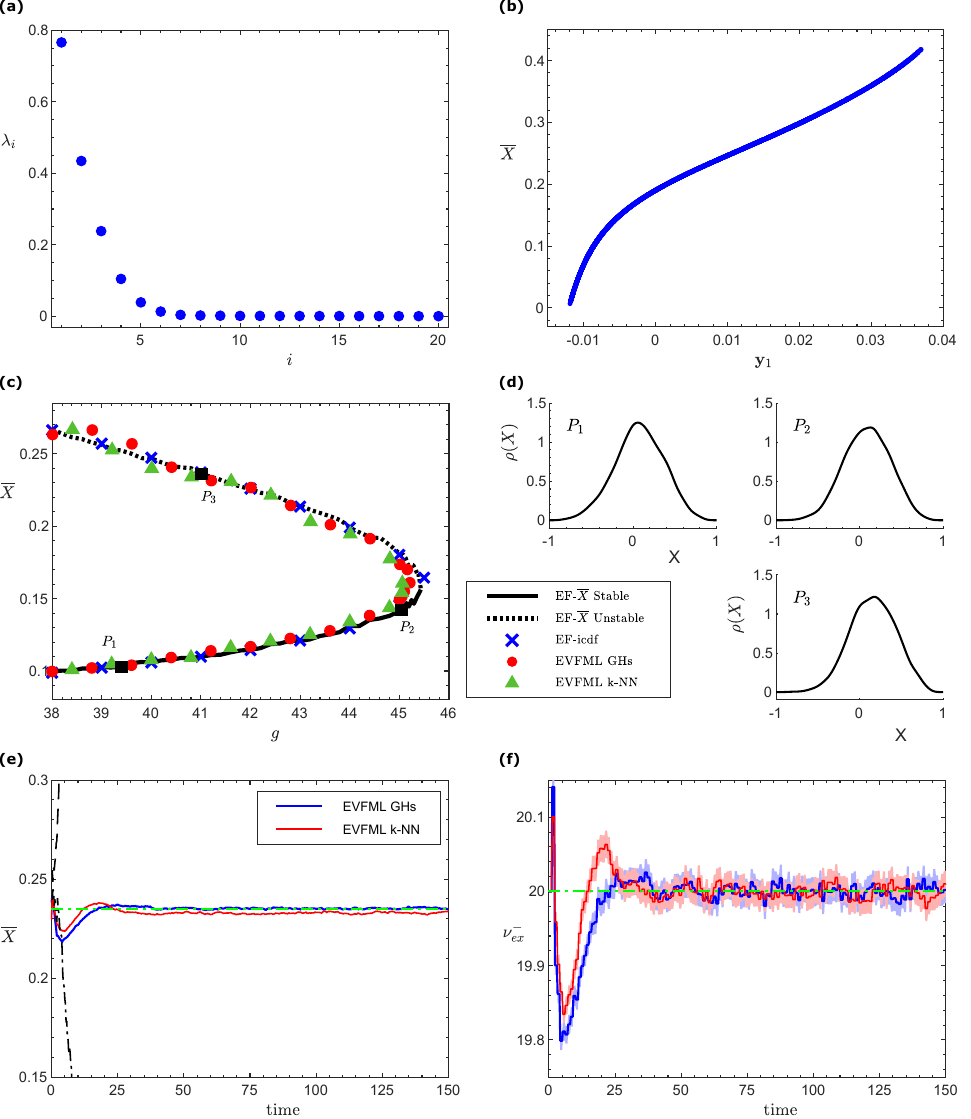}
	\caption{Stochastic agent-based model of a financial market with mimesis. (a) The first 20 eigenvalues of the DMs coordinates over the bifurcation parameter space, and (b) the corresponding first DMs coordinate, plotted against the mean value of the distribution $\overline{X}$. (c) Coarse-grained bifurcation diagrams constructed through the \emph{EVFML} framework by considering the first DMs coordinate $\mathbf{y}_1$ as the coarse-grained variable; for lifting both GHs and k-NN were considered.  Solid (dotted) lines denote stable (unstable) steady-states. ~For comparison purposes, we also depict the coarse-grained bifurcation diagram constructed using $\overline{X}$ (with black color) as coarse-grained variable that was discovered to be in a one-to-one correspondence with the first coordinate $\mathbf{y}_1$ of the DMs analysis, as well as the one constructed using the ICDF of the microscopic distribution and $R^{\pm}$ as coarse-grained variables (\citep{siettos2012equation}) on the basis of $50,000$ agents (with blue color).~(d) Density profiles of the distribution at selected points ($P_1$-$P_3$) along the bifurcation branch. (e),(f) Control of the stochastic agent-based financial market dynamics. Implementation of the wash-out embedded controller for stabilizing the coarse-grained unstable equilibrium  at $g =41$, using $\nu_{ex}^-$ as control variable.~(e) Coarse-grained open-loop and closed-loop responses with respect to $\overline{X}$, and (f) the responses of the control variable $\nu_{ex}^-(t)$. The controller drives and stabilizes the emergent dynamics to the actual open loop unstable steady-state.~Due to stochasticity, mean values (solid lines) and 95\% CI (shaded regions around them) are also reported based on $2,000$ runs.}
	\label{fig:TiM_AIO}
\end{figure}
We calculated the first 20 leading DMs by selecting $\epsilon=6.79$, thus retaining a high similarity between observations ($\sim 50\%$ of the pairwise distances $d_{ij}$ is greater than $\epsilon$).~The related eigenvalues $\lambda_i$ are displayed in Fig.~\ref{fig:TiM_AIO}(a), indicating a large spectral gap between $\lambda_1$ and the other eigenvalues.~Thus, $\mathbf{y}_1$ appears to be a good/adequate candidate for parametrizing the manifold.\par
As depicted in Fig.~\ref{fig:TiM_AIO}(b), $\mathbf{y}_1$ is related one-to-one with the first-order moment (mean value) $\overline{X}=E[X_i]$ of the microscopic distribution.~Higher-order moments did not reveal any clear relation with the other DMs embeddings.
~In addition, as Fig.~\ref{fig:TiM_AIO}(b) shows, the geometry arising from $\mathbf{y}_1$ in relation to $\overline{X}$ resembles the ICDF of $X_i$, which was utilized in \citep{siettos2012equation} for the construction of the coarse-grained bifurcation diagram.~Hence, the first DMs coordinate $\mathbf{y}_1$ provides an appropriate coarse-grained variable for the parametrization of the underlying manifold.\par
The numerical approximation performance of the restriction $\mathcal{R}$ and the lifting $\mathcal{L}$ operators over the bifurcation parameter space is presented in the Supplementary Material S2.1. ~Using a LOOCV, it is shown that the proposed GHs scheme provides a higher numerical approximation accuracy when compared to the k-NN algorithm.

\subsection{\textbf{Step B}. Equation/Variable-free numerical bifurcation analysis}
\label{subsub:BifD_TiM}

We proceeded with the construction of the \emph{coarse-timestepper} (Eq.~\eqref{eq:CTS}) with the coarse-grained variable being $\mathbf{y}_1$, thus ``wrapping'' around it the pseudo-arc-length continuation scheme.~Since the financial market model is stochastic, for each initial condition, we considered the average value of $2,000$ runs.~For the numerical simulations, we considered a healing step of $t_{skip}=8 t_{rep}$ and an evolution time step (see Eq.~\eqref{eq:CTS_imp}) of $\delta=4 t_{rep}$, while the continuation step sizes were set as $\delta g=0.05$ around the turning point.~The coarse-grained bifurcation diagrams as obtained by employing the \emph{EVFML} approach using the GHs and k-NN algorithms for lifting are shown in Fig.~\ref{fig:TiM_AIO}(c,d).\par
For comparison purposes, we constructed the coarse-grained bifurcation diagram with the aid of the EF framework by considering $\overline{X}$ as coarse-grained variable exploiting its one-to-one correspondence with the $\mathbf{y}_1$ coordinate that was discovered via the DMs analysis.
~In order to ensure consistent realizations, we performed lifting by sampling the internal agents' states from a Gaussian distribution of mean $\mu=E[X_i]$ and a fixed standard deviation $\sigma=0.32$.~For constructing the \emph{coarse-timestepper}, we selected the healing step $t_{skip}=4 t_{rep}$, since the system was shown to ``heal'' faster with this lifting operator.~All the other parameters were kept the same.
~The resulting bifurcation diagrams are superimposed in Fig.~\ref{fig:TiM_AIO}(c) along with the one derived in \citep{siettos2012equation} using the ICDF and the $R^{\pm}$ as coarse-grained variables for $50,000$ agents.
As Fig.~\ref{fig:TiM_AIO}(c) shows, the bifurcation diagrams derived using the $\mathbf{y}_1$ DMs coordinate and its one-to-one correspondence $\overline{X}$ as coarse-grained variable are almost identical to the one derived in \citep{siettos2012equation} where it was assumed a prior detailed knowledge of the physics.~Both stable and unstable branches were accurately detected, with the turning point estimated at ($\overline{X} \approx 0.1562, g\approx 45.43)$ when using $\overline{X}$ as coarse-grained variable and at $g \approx 45.21$ and $g \approx 45.06$ when using the DMs embedding as coarse-grained variable, coupled with GHs and k-NN for the lifting operator, respectively.

\subsection{\textbf{Step C}. Equation/Variable-free control}
\label{sub:EF_cont_TiM} 

From the numerical bifurcation analysis, we now have all the numerical quantities required for the design of linear controllers on the embedded space (the Jacobian matrix $\mathbf{A}$ and the control matrix $\mathbf{B}$).~For our illustrations, we have designed wash-out controllers to stabilize the emergent unstable open-loop equilibrium at $g = 41$ (see Fig.~\ref{fig:TiM_AIO}(c)), by selecting the exogenous arrival rate of ``bad'' information $u_0=\nu_{ex}^-$ as the control variable. 
~Without control, even if one starts exactly on the actual emergent unstable equilibrium, due to the inherent stochasticity, the system either blows-up or goes towards the emergent stable solution branch (dashed and dashed-dotted trajectories in Fig~\ref{fig:TiM_AIO}(e), respectively).

\subsubsection{Reconstruction of the low-dimensional manifold over the control parameter space}
\label{subsub:DM_TiM_Cont}
In order to increase the numerical accuracy for the implementation of the embedded controller, we repeated the procedure of step A over the control parameter space around the embedded steady-states that we seek to stabilize. Thus, we fixed $g=41$ and varied the control variable $\nu_{ex}^-$ around its nominal value to create a denser data set around the steady-state of interest.~Initial conditions were generated according to Eq.~\eqref{eq:ICs_TiM}; we considered $\mu$ to vary in an interval appropriate for generating various trajectories close to the unstable equilibrium.~For our computations, the new data set for the design of the embedded controller consisted of $31 \times 11 = 341$ trajectories, in the interval $t\in[0,40 t_{rep}]$ with a sampling time of reporting horizon $t_{rep}$, initialized with a value of $\mu\in [0.2, 0.4]$ every $d\mu=0.02$ and with variation of the control parameter $\nu_{ex}^-\in [19,21]$ with $d\nu_{ex}^-=0.066$.~In the cases where the solutions were blowing up (i.e., when $\overline{X}>0.5$) before $40 t_{rep}$, we kept the trajectories up to this blowing up limit.~Hence, the resulting new data set $\mathbf{X}\in\mathbb{R}^{M \times N}$ contained $M=8343$ observations of $N=5000$ internal states of the agents. We computed the first 20 DMs embeddings, this time by considering $\epsilon=8.10$.~As reported in the Supplementary Material S2.2, the first leading DMs coordinate $\mathbf{y}_1$ appears to be an appropriate coarse-grained variable, since (i) there is a large spectral gap between $\lambda_1$ and $\lambda_2$ (see Fig.~S5(a)), and (ii) $\mathbf{y}_1$ is in one-to-one relation with the first-order moment $\overline{X}=E[X_i]$ (see Fig.~S5(b)).\par
The numerical approximation performance of the restriction $\mathcal{R}$ and the lifting $\mathcal{L}$ operators over the control parameter space was assessed using a randomly distributed subset of $5,000$ observations drawn from the new local data set $\mathbf{X}$. A comparison with the results obtained with the k-NN algorithm is also provided.~As demonstrated in the Supplementary Material S2.2, following a LOOCV, the proposed GHs scheme results to a higher numerical approximation accuracy compared to the k-NN algorithm.

\subsubsection{Design and implementation of the embedded wash-out controller}
Based on the constructed \emph{coarse-timestepper}, we designed and implemented a wash-out controller using the first leading DM coordinate $\mathbf{y}_1$.~The sampling time/reporting horizon was set to $T=4 t_{rep}$.~Based on LQR, we obtained the optimal gain matrices $\mathbf{K}$ and $D$, which in the case of the lifting operator constructed via GHs are $\mathbf{K}=264.44$ and $D=0.159$, while in the case of k-NN are $\mathbf{K}=219.75$ and $D=0.176$.
The coarse-grained closed-loop response is shown in Fig.~\ref{fig:TiM_AIO}(e) in $\overline{X}$ coordinates after implementing the lifting operators with the GHs and k-NN algorithms; the mean values and the 95\% CI are displayed on the basis of the $2,000$ runs.~As shown in  Fig.~\ref{fig:TiM_AIO}(e), the \emph{EVFML} controllers successfully drive the system to the actual emergent open-loop unstable equilibrium.
~As it is shown in Fig.~\ref{fig:TiM_AIO}(f), the control response fluctuates slightly around its nominal value $\nu_{ex}^-=20$, due to the corresponding numerical approximation errors introduced by the numerical solution of the out-of-sample extension and pre-image problems and due to the inherent stochasticity.~Comparing the two different lifting operators the proposed GHs scheme results in smaller fluctuations of the control response $\nu_{ex}^-(t)$ around its nominal value; both lifting operators introduce relative changes less than 1\%  with respect to the nominal value, which become even smaller ($<0.1$\%) after $t>30$. 
%
%
%
%
\section{Conclusions}
\label{sec:con}
Recent theoretical and technological advances in machine and deep learning \cite{karniadakis2021physics} have revolutionized the way we model \cite{chen1990non,rico1992discrete,gonzalez1998identification,raissi2018hidden,lee2020coarse,lu2021learning,chen2021solving,SifanWang,galaris2022numerical,vlachas2022multiscale}, solve \cite{raissi2018hidden,han2018solving,raissi2019physics,cai2020deep,chen2021solving,dong2021local,schiassi2021extreme,fabiani2021numerical,calabro2021extreme,lu2021deepxde,dong2022computing}, analyze \cite{fabiani2021numerical,galaris2022numerical} and control high-dimensional multiscale/complex systems with applications ranging from complex fluids and materials \citep{fan2020reinforcement,zeng2022data,sirignano2020dpm,lye2020deep,vlachas2022multiscale,zhang2022analyses} and process control \citep{siettos2002truncated,siettos2002semiglobal,aggelogiannaki2008nonlinear,wu2019real,wu2019a,wu2019b,vaupel2020accelerating} to the biological and biomedical sciences \citep{peng2021multiscale,lovelett2019partial} and from environmental engineering and climate change \citep{rasp2018deep,bury2021deep} to sustainable mobility and robotics \cite{harel2020autonomics,lombardi2021using}. 
The high-dimensionality that intrinsically characterizes the state-space of relevant multiscale/complex problems, compounded by the inherent modelling uncertainties across scales, severely challenge our ability to efficiently understand, learn, analyze and control their collective behavior. \citep{karniadakis2021physics,peng2021multiscale}. 
Uncertainties range from the selection of an appropriate set of macroscopic variables for the description of the emergent behavior, to the closures that are required to bridge the micro- and macro- scales and allow the construction of reduced-order surrogate machine-learning models (see e.g. \citep{lee2020coarse,chen2021solving,galaris2022numerical,vlachas2022multiscale}) and/or learning the non-linear differential operators using for example DeepOnet \cite{lu2021learning} or Fourier Neural Operator for parametric PDEs \citep{li2020fourier}. For example, for the extraction of PDEs from data, 
challenges remain across algorithmic identification/learning steps pertaining to: (a) the ``correct'' structure of the evolution equation for the emergent dynamics (e.g., the inclusion or not of integral terms) \cite{li2007deciding},
(b) higher-order spatial derivatives from strongly noisy microscopic simulations, (c) the correct coarse-grained boundary conditions for problems with non-trivial dynamics such as the ones of the stochastic agent-based model that we considered here, and/or (d) the type of the ``color'' of the noise when comes to the identification of SDEs. 
{\em Ad-hoc} assumptions taken for the solution of the above problems may introduce qualitative and quantitative biases or even failures in the learning, analysis and control of the emergent dynamics.\par
Here, we demonstrated a data-driven equation-free/variable-free framework, integrating the Equation-free multiscale framework with Machine Learning, and in particular with non-linear manifold learning and Geometric Harmonics, to perform dimensionality reduction and control of the collective dynamics of complex systems whose dynamics are described by microscopic/agent-based simulators.
Our approach does not require an \emph{a-priori} knowledge of the ``correct'' macroscopic observables, nor physical insights on the ``correct'' type of PDEs, nor, importantly, on the extensive data collection necessary for the construction of global surrogate reduced-order machine-learning models. 
The learning of the coarse-dynamics is achieved locally, ``on demand'', using the agent-based simulations to identify a set of coarse-grained macroscopic variables both globally and locally, depending on the task and the numerical quantities involved in the design of linear controllers. 
The controllers are designed on an intrinsic manifold, and act directly on the high-dimensional space. For our illustrations, we have chosen to design and implement wash-out linear controllers. We also demonstrate that the scheme is robust against modelling uncertainties that are introduced from the construction of the coarse-timestepper related to the solution of the out-of-sample extension and pre-image problems.\par
Finally, we note that for many real-world complex problems, especially in biomedicine, or in social and financial sciences, high-fidelity simulations such as the ones  produced via detailed microscopic/agent-based models are simply not available. In such cases, one has to also deal with issues such as data sparsity and strong/varying ``colored'' stochasticity. This further impedes our ability to efficiently identify relevant data-driven manifolds \citep{gajamannage2019nonlinear} and consequently to the design of embedded controllers. Working at such a data-sparse limit remains a major open challenge; exploiting
local ML models to design informative data collection -as part of overall solution process- fits the ``on demand" structure of equation-free computation. We hope that linking the approach presented here with adaptive data collection and processing may lead to systematic construction of an ``exploration geometry". 


\begin{thebibliography}{92}
\providecommand{\natexlab}[1]{#1}
\providecommand{\url}[1]{\texttt{#1}}
\expandafter\ifx\csname urlstyle\endcsname\relax
  \providecommand{\doi}[1]{doi: #1}\else
  \providecommand{\doi}{doi: \begingroup \urlstyle{rm}\Url}\fi

\bibitem[Kevrekidis et~al.(2003)Kevrekidis, Gear, Hyman, Kevrekidis, Runborg,
  Theodoropoulos, et~al.]{kevrekidis2003equation}
I.~G. Kevrekidis, C.~W. Gear, J.~M. Hyman, P.~G. Kevrekidis, O.~Runborg,
  C.~Theodoropoulos, et~al.
\newblock Equation-free, coarse-grained multiscale computation: enabling
  microscopic simulators to perform system-level analysis.
\newblock \emph{Commun. Math. Sci}, 1\penalty0 (4):\penalty0 715--762, 2003.

\bibitem[Liu and Barab{\'a}si(2016)]{liu2016control}
Y.-Y. Liu and A.-L. Barab{\'a}si.
\newblock Control principles of complex systems.
\newblock \emph{Reviews of Modern Physics}, 88\penalty0 (3):\penalty0 035006,
  2016.

\bibitem[di~Bernardo(2020)]{diBernardo2020}
M.~di~Bernardo.
\newblock \emph{Controlling Collective Behavior in Complex Systems}, pages
  1--10.
\newblock Springer London, London, 2020.

\bibitem[Karniadakis et~al.(2021)Karniadakis, Kevrekidis, Lu, Perdikaris, Wang,
  and Yang]{karniadakis2021physics}
G.~E. Karniadakis, I.~G. Kevrekidis, L.~Lu, P.~Perdikaris, S.~Wang, and
  L.~Yang.
\newblock Physics-informed machine learning.
\newblock \emph{Nature Reviews Physics}, 3\penalty0 (6):\penalty0 422--440,
  2021.

\bibitem[Coifman et~al.(2005)Coifman, Lafon, Lee, Maggioni, Nadler, Warner, and
  Zucker]{coifman2005geometric}
R.~R. Coifman, S.~Lafon, A.~B. Lee, M.~Maggioni, B.~Nadler, F.~Warner, and
  S.~W. Zucker.
\newblock Geometric diffusions as a tool for harmonic analysis and structure
  definition of data: diffusion maps.
\newblock \emph{Proceedings of the National Academy of Sciences}, 102\penalty0
  (21):\penalty0 7426--7431, 2005.

\bibitem[Nadler et~al.(2006)Nadler, Lafon, Coifman, and
  Kevrekidis]{nadler2006diffusion}
B.~Nadler, S.~Lafon, R.~R. Coifman, and I.~G. Kevrekidis.
\newblock Diffusion maps, spectral clustering and reaction coordinates of
  dynamical systems.
\newblock \emph{Applied and Computational Harmonic Analysis}, 21\penalty0
  (1):\penalty0 113--127, 2006.

\bibitem[Coifman et~al.(2008)Coifman, Kevrekidis, Lafon, Maggioni, and
  Nadler]{coifman2008diffusion}
R.~R. Coifman, I.~G. Kevrekidis, S.~Lafon, M.~Maggioni, and B.~Nadler.
\newblock Diffusion maps, reduction coordinates, and low dimensional
  representation of stochastic systems.
\newblock \emph{Multiscale Modeling \& Simulation}, 7\penalty0 (2):\penalty0
  842--864, 2008.

\bibitem[Singer et~al.(2009)Singer, Erban, Kevrekidis, and
  Coifman]{singer2009detecting}
A.~Singer, R.~Erban, I.~G. Kevrekidis, and R.~R. Coifman.
\newblock Detecting intrinsic slow variables in stochastic dynamical systems by
  anisotropic diffusion maps.
\newblock \emph{Proceedings of the National Academy of Sciences}, 106\penalty0
  (38):\penalty0 16090--16095, 2009.

\bibitem[Chiavazzo et~al.(2014)Chiavazzo, Gear, Dsilva, Rabin, and
  Kevrekidis]{chiavazzo2014reduced}
E.~Chiavazzo, C.~W. Gear, C.~J. Dsilva, N.~Rabin, and I.~G. Kevrekidis.
\newblock Reduced models in chemical kinetics via nonlinear data-mining.
\newblock \emph{Processes}, 2\penalty0 (1):\penalty0 112--140, 2014.

\bibitem[Liu et~al.(2015)Liu, Siettos, Gear, and Kevrekidis]{liu2015equation}
P.~Liu, C.~I. Siettos, C.~W. Gear, and I.~G. Kevrekidis.
\newblock Equation-free model reduction in agent-based computations:
  Coarse-grained bifurcation and variable-free rare event analysis.
\newblock \emph{Mathematical Modelling of Natural Phenomena}, 10\penalty0
  (3):\penalty0 71--90, 2015.

\bibitem[Koronaki et~al.(2020)Koronaki, Nikas, and Boudouvis]{koronaki2020data}
E.~Koronaki, A.~Nikas, and A.~Boudouvis.
\newblock A data-driven reduced-order model of nonlinear processes based on
  diffusion maps and artificial neural networks.
\newblock \emph{Chemical Engineering Journal}, 397:\penalty0 125475, 2020.

\bibitem[Lee et~al.(2020)Lee, Kooshkbaghi, Spiliotis, Siettos, and
  Kevrekidis]{lee2020coarse}
S.~Lee, M.~Kooshkbaghi, K.~Spiliotis, C.~I. Siettos, and I.~G. Kevrekidis.
\newblock Coarse-scale pdes from fine-scale observations via machine learning.
\newblock \emph{Chaos: An Interdisciplinary Journal of Nonlinear Science},
  30\penalty0 (1):\penalty0 013141, 2020.

\bibitem[Galaris et~al.(2022)Galaris, Fabiani, Gallos, Kevrekidis, and
  Siettos]{galaris2022numerical}
E.~Galaris, G.~Fabiani, I.~Gallos, I.~Kevrekidis, and C.~Siettos.
\newblock Numerical bifurcation analysis of pdes from lattice boltzmann model
  simulations: a parsimonious machine learning approach.
\newblock \emph{Journal of Scientific Computing}, 92\penalty0 (2):\penalty0
  1--30, 2022.

\bibitem[Balasubramanian et~al.(2002)Balasubramanian, Schwartz, Tenenbaum,
  de~Silva, and Langford]{balasubramanian2002isomap}
M.~Balasubramanian, E.~L. Schwartz, J.~B. Tenenbaum, V.~de~Silva, and J.~C.
  Langford.
\newblock The {I}somap algorithm and topological stability.
\newblock \emph{Science}, 295\penalty0 (5552):\penalty0 7--7, 2002.

\bibitem[Bollt(2007)]{bollt2007attractor}
E.~Bollt.
\newblock Attractor modeling and empirical nonlinear model reduction of
  dissipative dynamical systems.
\newblock \emph{International Journal of Bifurcation and Chaos}, 17\penalty0
  (04):\penalty0 1199--1219, 2007.

\bibitem[Bhattacharjee and Matou{\v{s}}(2016)]{bhattacharjee2016nonlinear}
S.~Bhattacharjee and K.~Matou{\v{s}}.
\newblock A nonlinear manifold-based reduced order model for multiscale
  analysis of heterogeneous hyperelastic materials.
\newblock \emph{Journal of Computational Physics}, 313:\penalty0 635--653,
  2016.

\bibitem[Roweis and Saul(2000)]{roweis2000nonlinear}
S.~T. Roweis and L.~K. Saul.
\newblock Nonlinear dimensionality reduction by locally linear embedding.
\newblock \emph{science}, 290\penalty0 (5500):\penalty0 2323--2326, 2000.

\bibitem[Papaioannou et~al.(2022)Papaioannou, Talmon, Kevrekidis, and
  Siettos]{papaioannou2021time}
P.~Papaioannou, R.~Talmon, I.~Kevrekidis, and C.~Siettos.
\newblock Time series forecasting using manifold learning, radial basis
  function interpolation and geometric harmonics.
\newblock \emph{Chaos, to appear;arXiv preprint arXiv:2110.03625}, 2022.

\bibitem[Chen and Ferguson(2018)]{chen2018molecular}
W.~Chen and A.~L. Ferguson.
\newblock Molecular enhanced sampling with autoencoders: On-the-fly collective
  variable discovery and accelerated free energy landscape exploration.
\newblock \emph{Journal of computational chemistry}, 39\penalty0 (25):\penalty0
  2079--2102, 2018.

\bibitem[Vlachas et~al.(2022)Vlachas, Arampatzis, Uhler, and
  Koumoutsakos]{vlachas2022multiscale}
P.~R. Vlachas, G.~Arampatzis, C.~Uhler, and P.~Koumoutsakos.
\newblock Multiscale simulations of complex systems by learning their effective
  dynamics.
\newblock \emph{Nature Machine Intelligence}, 4\penalty0 (4):\penalty0
  359--366, 2022.

\bibitem[Brunton et~al.(2016{\natexlab{a}})Brunton, Proctor, and
  Kutz]{brunton2016discovering}
S.~L. Brunton, J.~L. Proctor, and J.~N. Kutz.
\newblock Discovering governing equations from data by sparse identification of
  nonlinear dynamical systems.
\newblock \emph{Proceedings of the national academy of sciences}, 113\penalty0
  (15):\penalty0 3932--3937, 2016{\natexlab{a}}.

\bibitem[Koopman and Neumann(1932)]{koopman1932dynamical}
B.~O. Koopman and J.~v. Neumann.
\newblock Dynamical systems of continuous spectra.
\newblock \emph{Proceedings of the National Academy of Sciences}, 18\penalty0
  (3):\penalty0 255--263, 1932.

\bibitem[Mezi{\'c}(2013)]{mezic2013analysis}
I.~Mezi{\'c}.
\newblock Analysis of fluid flows via spectral properties of the {K}oopman
  operator.
\newblock \emph{Annual Review of Fluid Mechanics}, 45:\penalty0 357--378, 2013.

\bibitem[Williams et~al.(2015)Williams, Kevrekidis, and
  Rowley]{williams2015data}
M.~O. Williams, I.~G. Kevrekidis, and C.~W. Rowley.
\newblock A data--driven approximation of the koopman operator: Extending
  dynamic mode decomposition.
\newblock \emph{Journal of Nonlinear Science}, 25\penalty0 (6):\penalty0
  1307--1346, 2015.

\bibitem[Brunton et~al.(2016{\natexlab{b}})Brunton, Brunton, Proctor, and
  Kutz]{brunton2016koopman}
S.~L. Brunton, B.~W. Brunton, J.~L. Proctor, and J.~N. Kutz.
\newblock Koopman invariant subspaces and finite linear representations of
  nonlinear dynamical systems for control.
\newblock \emph{PLOS ONE}, 11\penalty0 (2):\penalty0 1--19, 02
  2016{\natexlab{b}}.

\bibitem[Dietrich et~al.(2020)Dietrich, Thiem, and
  Kevrekidis]{dietrich2020koopman}
F.~Dietrich, T.~N. Thiem, and I.~G. Kevrekidis.
\newblock On the {K}oopman operator of algorithms.
\newblock \emph{SIAM Journal on Applied Dynamical Systems}, 19\penalty0
  (2):\penalty0 860--885, 2020.

\bibitem[Mauroy et~al.(2020)Mauroy, Susuki, and Mezi{\'c}]{mauroy2020koopman}
A.~Mauroy, Y.~Susuki, and I.~Mezi{\'c}.
\newblock \emph{Koopman operator in systems and control}.
\newblock Springer, 2020.

\bibitem[Raissi and Karniadakis(2018)]{raissi2018hidden}
M.~Raissi and G.~E. Karniadakis.
\newblock Hidden physics models: {M}achine learning of nonlinear partial
  differential equations.
\newblock \emph{Journal of Computational Physics}, 357:\penalty0 125--141,
  2018.

\bibitem[Raissi et~al.(2019)Raissi, Perdikaris, and
  Karniadakis]{raissi2019physics}
M.~Raissi, P.~Perdikaris, and G.~E. Karniadakis.
\newblock Physics-informed neural networks: {A} deep learning framework for
  solving forward and inverse problems involving nonlinear partial differential
  equations.
\newblock \emph{Journal of Computational physics}, 378:\penalty0 686--707,
  2019.

\bibitem[Vlachas et~al.(2018)Vlachas, Byeon, Wan, Sapsis, and
  Koumoutsakos]{vlachas2018data}
P.~R. Vlachas, W.~Byeon, Z.~Y. Wan, T.~P. Sapsis, and P.~Koumoutsakos.
\newblock Data-driven forecasting of high-dimensional chaotic systems with long
  short-term memory networks.
\newblock \emph{Proceedings of the Royal Society A: Mathematical, Physical and
  Engineering Sciences}, 474\penalty0 (2213):\penalty0 20170844, 2018.

\bibitem[Theodoropoulos et~al.(2000)Theodoropoulos, Qian, and
  Kevrekidis]{theodoropoulos2000coarse}
C.~Theodoropoulos, Y.-H. Qian, and I.~G. Kevrekidis.
\newblock “coarse” stability and bifurcation analysis using time-steppers:
  A reaction-diffusion example.
\newblock \emph{Proceedings of the National Academy of Sciences}, 97\penalty0
  (18):\penalty0 9840--9843, 2000.

\bibitem[Kevrekidis et~al.(2004)Kevrekidis, Gear, and
  Hummer]{kevrekidis2004equation}
I.~G. Kevrekidis, C.~W. Gear, and G.~Hummer.
\newblock Equation-free: The computer-aided analysis of complex multiscale
  systems.
\newblock \emph{AIChE Journal}, 50\penalty0 (7):\penalty0 1346--1355, 2004.

\bibitem[Zagaris et~al.(2009)Zagaris, Gear, Kaper, and
  Kevrekidis]{zagaris2009analysis}
A.~Zagaris, C.~W. Gear, T.~J. Kaper, and Y.~G. Kevrekidis.
\newblock Analysis of the accuracy and convergence of equation-free projection
  to a slow manifold.
\newblock \emph{ESAIM: Mathematical Modelling and Numerical Analysis},
  43\penalty0 (4):\penalty0 757--784, 2009.

\bibitem[Siettos and Russo(2022)]{siettos2022numerical}
C.~Siettos and L.~Russo.
\newblock A numerical method for the approximation of stable and unstable
  manifolds of microscopic simulators.
\newblock \emph{Numerical Algorithms}, 89\penalty0 (3):\penalty0 1335--1368,
  2022.

\bibitem[Maclean et~al.(2021)Maclean, Bunder, and Roberts]{maclean2021toolbox}
J.~Maclean, J.~Bunder, and A.~J. Roberts.
\newblock A toolbox of equation-free functions in matlab/octave for efficient
  system level simulation.
\newblock \emph{Numerical Algorithms}, 87\penalty0 (4):\penalty0 1729--1748,
  2021.

\bibitem[Siettos et~al.(2004)Siettos, Kevrekidis, and
  Maroudas]{siettos2004coarse}
C.~I. Siettos, I.~G. Kevrekidis, and D.~Maroudas.
\newblock Coarse bifurcation diagrams via microscopic simulators: a
  state-feedback control-based approach.
\newblock \emph{International Journal of Bifurcation and Chaos}, 14\penalty0
  (01):\penalty0 207--220, 2004.

\bibitem[Sieber et~al.(2011)Sieber, Krauskopf, Wagg, Neild, and
  Gonzalez-Buelga]{sieber2011control}
J.~Sieber, B.~Krauskopf, D.~Wagg, S.~Neild, and A.~Gonzalez-Buelga.
\newblock Control-based continuation of unstable periodic orbits.
\newblock \emph{Journal of Computational and Nonlinear Dynamics}, 6\penalty0
  (1), 2011.

\bibitem[Barton and Sieber(2013)]{barton2013systematic}
D.~A. Barton and J.~Sieber.
\newblock Systematic experimental exploration of bifurcations with noninvasive
  control.
\newblock \emph{Physical Review E}, 87\penalty0 (5):\penalty0 052916, 2013.

\bibitem[Renson et~al.(2019)Renson, Sieber, Barton, Shaw, and
  Neild]{renson2019numerical}
L.~Renson, J.~Sieber, D.~A. Barton, A.~Shaw, and S.~Neild.
\newblock Numerical continuation in nonlinear experiments using local gaussian
  process regression.
\newblock \emph{Nonlinear Dynamics}, 98\penalty0 (4):\penalty0 2811--2826,
  2019.

\bibitem[Panagiotopoulos et~al.(2022)Panagiotopoulos, Starke, Sieber, and
  Just]{panagiotopoulos2022continuation}
I.~Panagiotopoulos, J.~Starke, J.~Sieber, and W.~Just.
\newblock Continuation with non-invasive control schemes: Revealing unstable
  states in a pedestrian evacuation scenario.
\newblock \emph{arXiv preprint arXiv:2203.02484}, 2022.

\bibitem[Chin et~al.(2022)Chin, Ruth, Sanford, Santorella, Carter, and
  Sandstede]{chin2022enabling}
T.~Chin, J.~Ruth, C.~Sanford, R.~Santorella, P.~Carter, and B.~Sandstede.
\newblock Enabling equation-free modeling via diffusion maps.
\newblock \emph{Journal of Dynamics and Differential Equations}, pages 1--20,
  2022.

\bibitem[Abed et~al.(1994)Abed, Wang, and Chen]{abed1994stabilization}
E.~H. Abed, H.~Wang, and R.~Chen.
\newblock Stabilization of period doubling bifurcations and implications for
  control of chaos.
\newblock \emph{Physica D: Nonlinear Phenomena}, 70\penalty0 (1-2):\penalty0
  154--164, 1994.

\bibitem[Siettos et~al.(2012)Siettos, Gear, and
  Kevrekidis]{siettos2012equation}
C.~I. Siettos, C.~W. Gear, and I.~G. Kevrekidis.
\newblock An equation-free approach to agent-based computation: Bifurcation
  analysis and control of stationary states.
\newblock \emph{EPL (Europhysics Letters)}, 99\penalty0 (4):\penalty0 48007,
  2012.

\bibitem[Marschler et~al.(2014{\natexlab{a}})Marschler, Starke, Liu, and
  Kevrekidis]{marschler2014coarse}
C.~Marschler, J.~Starke, P.~Liu, and I.~G. Kevrekidis.
\newblock Coarse-grained particle model for pedestrian flow using diffusion
  maps.
\newblock \emph{Physical Review E}, 89\penalty0 (1):\penalty0 013304,
  2014{\natexlab{a}}.

\bibitem[Omurtag and Sirovich(2006)]{omurtag2006modeling}
A.~Omurtag and L.~Sirovich.
\newblock Modeling a large population of traders: Mimesis and stability.
\newblock \emph{Journal of Economic Behavior \& Organization}, 61\penalty0
  (4):\penalty0 562--576, 2006.

\bibitem[Coifman and Lafon(2006)]{coifman2006geometric}
R.~R. Coifman and S.~Lafon.
\newblock Geometric harmonics: a novel tool for multiscale out-of-sample
  extension of empirical functions.
\newblock \emph{Applied and Computational Harmonic Analysis}, 21\penalty0
  (1):\penalty0 31--52, 2006.

\bibitem[Dsilva et~al.(2018)Dsilva, Talmon, Coifman, and
  Kevrekidis]{dsilva2018parsimonious}
C.~J. Dsilva, R.~Talmon, R.~R. Coifman, and I.~G. Kevrekidis.
\newblock Parsimonious representation of nonlinear dynamical systems through
  manifold learning: A chemotaxis case study.
\newblock \emph{Applied and Computational Harmonic Analysis}, 44\penalty0
  (3):\penalty0 759--773, 2018.

\bibitem[Holiday et~al.(2019)Holiday, Kooshkbaghi, Bello-Rivas, Gear, Zagaris,
  and Kevrekidis]{holiday2019manifold}
A.~Holiday, M.~Kooshkbaghi, J.~M. Bello-Rivas, C.~W. Gear, A.~Zagaris, and
  I.~G. Kevrekidis.
\newblock Manifold learning for parameter reduction.
\newblock \emph{Journal of computational physics}, 392:\penalty0 419--431,
  2019.

\bibitem[Nystr{\"o}m(1929)]{nystrom1929praktische}
E.~J. Nystr{\"o}m.
\newblock \emph{{\"U}ber die praktische aufl{\"o}sung von linearen
  integralgleichungen mit anwendungen auf randwertaufgaben der
  potentialtheorie}.
\newblock Akademische Buchhandlung, 1929.

\bibitem[Evangelou et~al.(2022)Evangelou, Dietrich, Chiavazzo, Lehmberg, Meila,
  and Kevrekidis]{evangelou2022double}
N.~Evangelou, F.~Dietrich, E.~Chiavazzo, D.~Lehmberg, M.~Meila, and I.~G.
  Kevrekidis.
\newblock Double diffusion maps and their latent harmonics for scientific
  computations in latent space.
\newblock \emph{arXiv preprint arXiv:2204.12536}, 2022.

\bibitem[Armaou et~al.(2004)Armaou, I.~Siettos, and
  G.~Kevrekidis]{armaou2004time}
A.~Armaou, C.~I.~Siettos, and I.~G.~Kevrekidis.
\newblock Time-steppers and ‘coarse’ control of distributed microscopic
  processes.
\newblock \emph{International Journal of Robust and Nonlinear Control},
  14\penalty0 (2):\penalty0 89--111, 2004.

\bibitem[Siettos et~al.(2006)Siettos, Kevrekidis, and
  Kazantzis]{siettos2006equation}
C.~I. Siettos, I.~G. Kevrekidis, and N.~Kazantzis.
\newblock An equation-free approach to nonlinear control: Coarse feedback
  linearization with pole-placement.
\newblock \emph{International Journal of Bifurcation and Chaos}, 16\penalty0
  (07):\penalty0 2029--2041, 2006.

\bibitem[Kevrekidis and Samaey(2009)]{kevrekidis2009equation}
I.~G. Kevrekidis and G.~Samaey.
\newblock Equation-free multiscale computation: Algorithms and applications.
\newblock \emph{Annual review of physical chemistry}, 60:\penalty0 321--344,
  2009.

\bibitem[Bauer and Fike(1960)]{bauer1960norms}
F.~L. Bauer and C.~T. Fike.
\newblock Norms and exclusion theorems.
\newblock \emph{Numerische Mathematik}, 2\penalty0 (1):\penalty0 137--141,
  1960.

\bibitem[Marschler et~al.(2014{\natexlab{b}})Marschler, Sieber, Berkemer,
  Kawamoto, and Starke]{marschler2014implicit}
C.~Marschler, J.~Sieber, R.~Berkemer, A.~Kawamoto, and J.~Starke.
\newblock Implicit methods for equation-free analysis: Convergence results and
  analysis of emergent waves in microscopic traffic models.
\newblock \emph{SIAM Journal on Applied Dynamical Systems}, 13\penalty0
  (3):\penalty0 1202--1238, 2014{\natexlab{b}}.

\bibitem[Bando et~al.(1995)Bando, Hasebe, Nakayama, Shibata, and
  Sugiyama]{bando1995dynamical}
M.~Bando, K.~Hasebe, A.~Nakayama, A.~Shibata, and Y.~Sugiyama.
\newblock Dynamical model of traffic congestion and numerical simulation.
\newblock \emph{Physical review E}, 51\penalty0 (2):\penalty0 1035, 1995.

\bibitem[Gaididei et~al.(2009)Gaididei, Berkemer, Caputo, Christiansen,
  Kawamoto, Shiga, S{\o}rensen, and Starke]{gaididei2009analytical}
Y.~B. Gaididei, R.~Berkemer, J.~G. Caputo, P.~Christiansen, A.~Kawamoto,
  T.~Shiga, M.~P. S{\o}rensen, and J.~Starke.
\newblock Analytical solutions of jam pattern formation on a ring for a class
  of optimal velocity traffic models.
\newblock \emph{New Journal of Physics}, 11\penalty0 (7):\penalty0 073012,
  2009.

\bibitem[Chen et~al.(1990)Chen, Billings, and Grant]{chen1990non}
S.~Chen, S.~A. Billings, and P.~Grant.
\newblock Non-linear system identification using neural networks.
\newblock \emph{International journal of control}, 51\penalty0 (6):\penalty0
  1191--1214, 1990.

\bibitem[Rico-Martinez et~al.(1992)Rico-Martinez, Krischer, Kevrekidis, Kube,
  and Hudson]{rico1992discrete}
R.~Rico-Martinez, K.~Krischer, I.~Kevrekidis, M.~Kube, and J.~Hudson.
\newblock Discrete-vs. continuous-time nonlinear signal processing of {C}u
  electrodissolution data.
\newblock \emph{Chemical Engineering Communications}, 118\penalty0
  (1):\penalty0 25--48, 1992.

\bibitem[Gonz{\'a}lez-Garc{\'\i}a et~al.(1998)Gonz{\'a}lez-Garc{\'\i}a,
  Rico-Mart{\`\i}nez, and Kevrekidis]{gonzalez1998identification}
R.~Gonz{\'a}lez-Garc{\'\i}a, R.~Rico-Mart{\`\i}nez, and I.~G. Kevrekidis.
\newblock Identification of distributed parameter systems: {A} neural net based
  approach.
\newblock \emph{Computers \& chemical engineering}, 22:\penalty0 S965--S968,
  1998.

\bibitem[Lu et~al.(2021{\natexlab{a}})Lu, Jin, Pang, Zhang, and
  Karniadakis]{lu2021learning}
L.~Lu, P.~Jin, G.~Pang, Z.~Zhang, and G.~E. Karniadakis.
\newblock Learning nonlinear operators via deeponet based on the universal
  approximation theorem of operators.
\newblock \emph{Nature Machine Intelligence}, 3\penalty0 (3):\penalty0
  218--229, 2021{\natexlab{a}}.

\bibitem[Chen et~al.(2021)Chen, Hosseini, Owhadi, and Stuart]{chen2021solving}
Y.~Chen, B.~Hosseini, H.~Owhadi, and A.~M. Stuart.
\newblock Solving and learning nonlinear pdes with gaussian processes.
\newblock \emph{Journal of Computational Physics}, 447:\penalty0 110668, 2021.

\bibitem[Wang et~al.(2021)Wang, Wang, and Perdikaris]{SifanWang}
S.~Wang, H.~Wang, and P.~Perdikaris.
\newblock Learning the solution operator of parametric partial differential
  equations with physics-informed deeponets.
\newblock \emph{Science Advances}, 7\penalty0 (40):\penalty0 eabi8605, 2021.
\newblock \doi{10.1126/sciadv.abi8605}.

\bibitem[Han et~al.(2018)Han, Jentzen, and E]{han2018solving}
J.~Han, A.~Jentzen, and W.~E.
\newblock Solving high-dimensional partial differential equations using deep
  learning.
\newblock \emph{Proceedings of the National Academy of Sciences}, 115\penalty0
  (34):\penalty0 8505--8510, 2018.

\bibitem[Cai et~al.(2020)Cai, Chen, Liu, and Liu]{cai2020deep}
Z.~Cai, J.~Chen, M.~Liu, and X.~Liu.
\newblock Deep least-squares methods: {A}n unsupervised learning-based
  numerical method for solving elliptic pdes.
\newblock \emph{Journal of Computational Physics}, 420:\penalty0 109707, 2020.

\bibitem[Dong and Li(2021)]{dong2021local}
S.~Dong and Z.~Li.
\newblock Local extreme learning machines and domain decomposition for solving
  linear and nonlinear partial differential equations.
\newblock \emph{Computer Methods in Applied Mechanics and Engineering},
  387:\penalty0 114129, 2021.

\bibitem[Schiassi et~al.(2021)Schiassi, Furfaro, Leake, De~Florio, Johnston,
  and Mortari]{schiassi2021extreme}
E.~Schiassi, R.~Furfaro, C.~Leake, M.~De~Florio, H.~Johnston, and D.~Mortari.
\newblock Extreme theory of functional connections: {A} fast physics-informed
  neural network method for solving ordinary and partial differential
  equations.
\newblock \emph{Neurocomputing}, 457:\penalty0 334--356, 2021.

\bibitem[Fabiani et~al.(2021)Fabiani, Calabr{\`o}, Russo, and
  Siettos]{fabiani2021numerical}
G.~Fabiani, F.~Calabr{\`o}, L.~Russo, and C.~Siettos.
\newblock Numerical solution and bifurcation analysis of nonlinear partial
  differential equations with extreme learning machines.
\newblock \emph{Journal of Scientific Computing}, 89\penalty0 (2):\penalty0
  1--35, 2021.

\bibitem[Calabr{\`o} et~al.(2021)Calabr{\`o}, Fabiani, and
  Siettos]{calabro2021extreme}
F.~Calabr{\`o}, G.~Fabiani, and C.~Siettos.
\newblock Extreme learning machine collocation for the numerical solution of
  elliptic {PDE}s with sharp gradients.
\newblock \emph{Computer Methods in Applied Mechanics and Engineering},
  387:\penalty0 114188, 2021.

\bibitem[Lu et~al.(2021{\natexlab{b}})Lu, Meng, Mao, and
  Karniadakis]{lu2021deepxde}
L.~Lu, X.~Meng, Z.~Mao, and G.~E. Karniadakis.
\newblock Deep{XDE}: {A} deep learning library for solving differential
  equations.
\newblock \emph{SIAM Review}, 63\penalty0 (1):\penalty0 208--228,
  2021{\natexlab{b}}.

\bibitem[Dong and Yang(2022)]{dong2022computing}
S.~Dong and J.~Yang.
\newblock On computing the hyperparameter of extreme learning machines:
  {A}lgorithm and application to computational {PDE}s, and comparison with
  classical and high-order finite elements.
\newblock \emph{Journal of Computational Physics}, page 111290, 2022.

\bibitem[Fan et~al.(2020)Fan, Yang, Wang, Triantafyllou, and
  Karniadakis]{fan2020reinforcement}
D.~Fan, L.~Yang, Z.~Wang, M.~S. Triantafyllou, and G.~E. Karniadakis.
\newblock Reinforcement learning for bluff body active flow control in
  experiments and simulations.
\newblock \emph{Proceedings of the National Academy of Sciences}, 117\penalty0
  (42):\penalty0 26091--26098, 2020.

\bibitem[Zeng et~al.(2022)Zeng, Linot, and Graham]{zeng2022data}
K.~Zeng, A.~J. Linot, and M.~D. Graham.
\newblock Data-driven control of spatiotemporal chaos with reduced-order neural
  ode-based models and reinforcement learning.
\newblock \emph{arXiv preprint arXiv:2205.00579}, 2022.

\bibitem[Sirignano et~al.(2020)Sirignano, MacArt, and Freund]{sirignano2020dpm}
J.~Sirignano, J.~F. MacArt, and J.~B. Freund.
\newblock D{PM}: {A} deep learning {PDE} augmentation method with application
  to large-eddy simulation.
\newblock \emph{Journal of Computational Physics}, 423:\penalty0 109811, 2020.

\bibitem[Lye et~al.(2020)Lye, Mishra, and Ray]{lye2020deep}
K.~O. Lye, S.~Mishra, and D.~Ray.
\newblock Deep learning observables in computational fluid dynamics.
\newblock \emph{Journal of Computational Physics}, 410:\penalty0 109339, 2020.

\bibitem[Zhang et~al.(2022)Zhang, Dao, Karniadakis, and
  Suresh]{zhang2022analyses}
E.~Zhang, M.~Dao, G.~E. Karniadakis, and S.~Suresh.
\newblock Analyses of internal structures and defects in materials using
  physics-informed neural networks.
\newblock \emph{Science advances}, 8\penalty0 (7):\penalty0 eabk0644, 2022.

\bibitem[Siettos et~al.(2002)Siettos, Bafas, and
  Boudouvis]{siettos2002truncated}
C.~I. Siettos, G.~V. Bafas, and A.~G. Boudouvis.
\newblock Truncated {C}hebyshev series approximation of fuzzy systems for
  control and nonlinear system identification.
\newblock \emph{Fuzzy sets and systems}, 126\penalty0 (1):\penalty0 89--104,
  2002.

\bibitem[Siettos and Bafas(2002)]{siettos2002semiglobal}
C.~I. Siettos and G.~V. Bafas.
\newblock Semiglobal stabilization of nonlinear systems using fuzzy control and
  singular perturbation methods.
\newblock \emph{Fuzzy Sets and Systems}, 129\penalty0 (3):\penalty0 275--294,
  2002.

\bibitem[Aggelogiannaki and Sarimveis(2008)]{aggelogiannaki2008nonlinear}
E.~Aggelogiannaki and H.~Sarimveis.
\newblock Nonlinear model predictive control for distributed parameter systems
  using data driven artificial neural network models.
\newblock \emph{Computers \& Chemical Engineering}, 32\penalty0 (6):\penalty0
  1225--1237, 2008.

\bibitem[Wu et~al.(2019{\natexlab{a}})Wu, Rincon, and Christofides]{wu2019real}
Z.~Wu, D.~Rincon, and P.~D. Christofides.
\newblock Real-time adaptive machine-learning-based predictive control of
  nonlinear processes.
\newblock \emph{Industrial \& Engineering Chemistry Research}, 59\penalty0
  (6):\penalty0 2275--2290, 2019{\natexlab{a}}.

\bibitem[Wu et~al.(2019{\natexlab{b}})Wu, Tran, Rincon, and
  Christofides]{wu2019a}
Z.~Wu, A.~Tran, D.~Rincon, and P.~D. Christofides.
\newblock Machine learning-based predictive control of nonlinear processes.
  part i: Theory.
\newblock \emph{AIChE Journal}, 65\penalty0 (11):\penalty0 e16729,
  2019{\natexlab{b}}.

\bibitem[Wu et~al.(2019{\natexlab{c}})Wu, Tran, Rincon, and
  Christofides]{wu2019b}
Z.~Wu, A.~Tran, D.~Rincon, and P.~D. Christofides.
\newblock Machine-learning-based predictive control of nonlinear processes.
  part ii: Computational implementation.
\newblock \emph{AIChE Journal}, 65\penalty0 (11):\penalty0 e16734,
  2019{\natexlab{c}}.

\bibitem[Vaupel et~al.(2020)Vaupel, Hamacher, Caspari, Mhamdi, Kevrekidis, and
  Mitsos]{vaupel2020accelerating}
Y.~Vaupel, N.~C. Hamacher, A.~Caspari, A.~Mhamdi, I.~G. Kevrekidis, and
  A.~Mitsos.
\newblock Accelerating nonlinear model predictive control through machine
  learning.
\newblock \emph{Journal of Process Control}, 92:\penalty0 261--270, 2020.

\bibitem[Peng et~al.(2021)Peng, Alber, Buganza~Tepole, Cannon, De, Dura-Bernal,
  Garikipati, Karniadakis, Lytton, Perdikaris, et~al.]{peng2021multiscale}
G.~C. Peng, M.~Alber, A.~Buganza~Tepole, W.~R. Cannon, S.~De, S.~Dura-Bernal,
  K.~Garikipati, G.~Karniadakis, W.~W. Lytton, P.~Perdikaris, et~al.
\newblock Multiscale modeling meets machine learning: What can we learn?
\newblock \emph{Archives of Computational Methods in Engineering}, 28\penalty0
  (3):\penalty0 1017--1037, 2021.

\bibitem[Lovelett et~al.(2019)Lovelett, Avalos, and
  Kevrekidis]{lovelett2019partial}
R.~J. Lovelett, J.~L. Avalos, and I.~G. Kevrekidis.
\newblock Partial observations and conservation laws: gray-box modeling in
  biotechnology and optogenetics.
\newblock \emph{Industrial \& Engineering Chemistry Research}, 59\penalty0
  (6):\penalty0 2611--2620, 2019.

\bibitem[Rasp et~al.(2018)Rasp, Pritchard, and Gentine]{rasp2018deep}
S.~Rasp, M.~S. Pritchard, and P.~Gentine.
\newblock Deep learning to represent subgrid processes in climate models.
\newblock \emph{Proceedings of the National Academy of Sciences}, 115\penalty0
  (39):\penalty0 9684--9689, 2018.

\bibitem[Bury et~al.(2021)Bury, Sujith, Pavithran, Scheffer, Lenton, Anand, and
  Bauch]{bury2021deep}
T.~M. Bury, R.~Sujith, I.~Pavithran, M.~Scheffer, T.~M. Lenton, M.~Anand, and
  C.~T. Bauch.
\newblock Deep learning for early warning signals of tipping points.
\newblock \emph{Proceedings of the National Academy of Sciences}, 118\penalty0
  (39):\penalty0 e2106140118, 2021.

\bibitem[Harel et~al.(2020)Harel, Marron, and Sifakis]{harel2020autonomics}
D.~Harel, A.~Marron, and J.~Sifakis.
\newblock Autonomics: In search of a foundation for next-generation autonomous
  systems.
\newblock \emph{Proceedings of the National Academy of Sciences}, 117\penalty0
  (30):\penalty0 17491--17498, 2020.

\bibitem[Lombardi et~al.(2021)Lombardi, Liuzza, and
  di~Bernardo]{lombardi2021using}
M.~Lombardi, D.~Liuzza, and M.~di~Bernardo.
\newblock Using learning to control artificial avatars in human motor
  coordination tasks.
\newblock \emph{IEEE Transactions on Robotics}, 37\penalty0 (6):\penalty0
  2067--2082, 2021.

\bibitem[Li et~al.(2020)Li, Kovachki, Azizzadenesheli, Liu, Bhattacharya,
  Stuart, and Anandkumar]{li2020fourier}
Z.~Li, N.~Kovachki, K.~Azizzadenesheli, B.~Liu, K.~Bhattacharya, A.~Stuart, and
  A.~Anandkumar.
\newblock Fourier neural operator for parametric partial differential
  equations.
\newblock \emph{arXiv preprint arXiv:2010.08895}, 2020.

\bibitem[Li et~al.(2007)Li, Kevrekidis, Gear, and Kevrekidis]{li2007deciding}
J.~Li, P.~G. Kevrekidis, C.~W. Gear, and I.~G. Kevrekidis.
\newblock Deciding the nature of the coarse equation through microscopic
  simulations: {T}he baby-bathwater scheme.
\newblock \emph{SIAM review}, 49\penalty0 (3):\penalty0 469--487, 2007.

\bibitem[Gajamannage et~al.(2019)Gajamannage, Paffenroth, and
  Bollt]{gajamannage2019nonlinear}
K.~Gajamannage, R.~Paffenroth, and E.~M. Bollt.
\newblock A nonlinear dimensionality reduction framework using smooth
  geodesics.
\newblock \emph{Pattern Recognition}, 87:\penalty0 226--236, 2019.

\end{thebibliography}

%
%
%
\clearpage
\newpage
\renewcommand{\thetable}{S\arabic{table}}  
\renewcommand{\thefigure}{S\arabic{figure}}
\renewcommand{\thesection}{S\arabic{section}} 
\setcounter{section}{0}
\section*{Supplementary Material}
\section{Agent-based Traffic Model}
\subsection{Step A. Numerical Approximation Accuracy of the Restriction and Lifting operators for the Coarse-grained Numerical Bifurcation Analysis}

Here, we report the numerical approximation accuracy of the restriction and lifting operators constructed on the basis of the DMs embeddings over the bifurcation parameter space of the traffic model using the leave-one-out cross-validation (LOOCV) procedure described in Section~\ref{subsub:LOOCV}.~The resulting numerical approximation errors are given in Table~\ref{tab:Norms_TrM_BD} in terms of three metrics, namely the $L_1$, $L_2$ and $L_{\infty}$ norms; 95\% confidence intervals (CI) are also provided.
\begin{table}[!h]
    \centering
    \caption{Numerical approximation accuracy in terms of $L_1$, $L_2$ and $L_{\infty}$ norms of the restriction operator (obtained by coupling DMS with the Nystr\"{o}m method) (left column) and the lifting operator (obtained by coupling DMs with GHs and k-NN) (middle and right columns).}
    \begin{tabular}{l| c c | c c | c c}
    \toprule
    Operator   & \multicolumn{2}{c}{Restriction $\mathcal{R}$} & \multicolumn{2}{|c}{Lifting $\mathcal{L}$ with GHs} & \multicolumn{2}{|c}{Lifting $\mathcal{L}$ with k-NN} \\
    \midrule
    Error norm & mean $\times 10^{-4}$ & 95\% CI $\times 10^{-4}$ & mean $\times 10^{-2}$ & 95\% CI $\times 10^{-2}$ & mean $\times 10^{-2}$ & 95\% CI $\times 10^{-2}$ \\
    \midrule
    $\lVert \cdot \rVert_1$ & 2.788  & (2.747, 2.829) & 7.338 & (7.193,7.482) & 8.896 & (8.728,9.064) \\
    $\lVert \cdot \rVert_2$ & 2.194  & (2.163, 2.226) & 2.087 & (2.042,2.127) & 2.603 & (2.553,2.654)  \\
    $\lVert \cdot \rVert_{\infty}$ & 1.979  & (1.950, 2.008) & 1.028 & (1.005,1.051) & 1.367 & (1.339,1.395) \\
    \bottomrule
    \end{tabular}
    \label{tab:Norms_TrM_BD}
\end{table}
\begin{figure}[!hb]
    \centering
    \includegraphics[width=\textwidth]{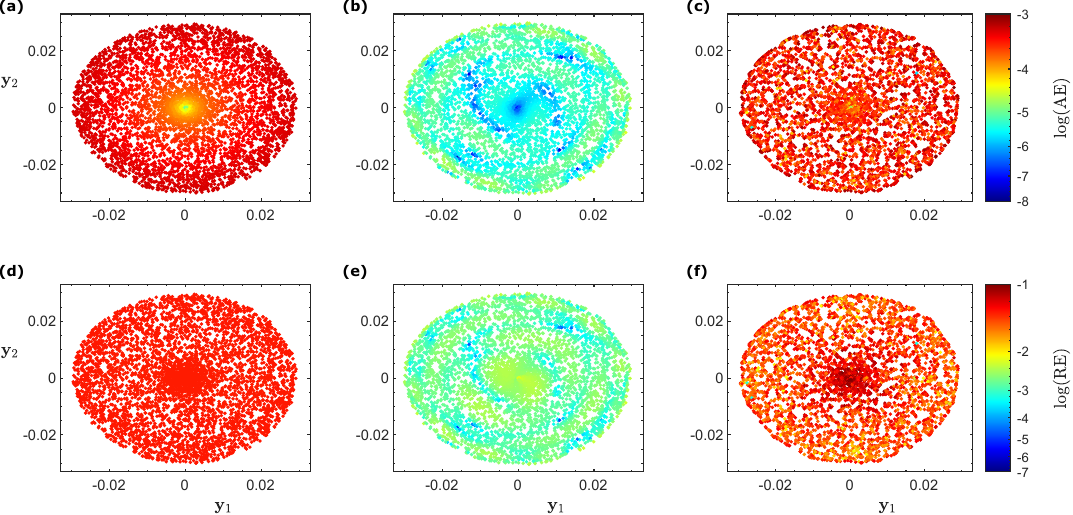}
    \caption{~Visualization of the numerical approximation accuracy of the restriction $\mathcal{R}$ operator and the restriction-lifting composition $\mathcal{R} \circ \mathcal{L}$  based on the data set of $\mathbf{X}\in \mathbb{R}^{5000 \times 20}$ points used to perform DMs over the bifurcation parameter space.~The numerical approximation accuracy was obtained through a LOOCV procedure as described in Section~\ref{subsub:LOOCV}.~Panels (a-c) and (d-f) depict the absolute and relative approximation $L_2$ norm errors, respectively.~Panels (a,d) correspond to the restriction operator (obtained by coupling DMs and the Nystr\"{o}m method), and panels (b,e) and (c,f) correspond to the restriction-lifting composition $\mathcal{R} (\mathcal{L})$ using GHs and k-NN for lifting, respectively.}
        \label{fig:testOp_bif_TrM}
\end{figure}
For implementing the lifting operator with GHs, we have set $\tilde{\epsilon}=0.025$ and $\delta= 2.17 \times 10^{-8}$.~This choice results to 60 eigenvectors of the ``second'' DMs embeddings.~For implementing the lifting operator with k-NN, $k=8$ neighbors were considered; no significant accuracy improvement was reported with more neighbors.~As shown in Table~\ref{tab:Norms_TrM_BD}, a high numerical approximation accuracy is reported for the restriction operator (obtained by coupling DMs with the Nystr\"{o}m method) and the lifting operator (obtained by coupling DMs with both GHs and k-NN algorithms).~As it is shown, lifting with GHs was more accurate than the one obtained with the k-NN algorithm.~The absolute and relative errors of the restriction $\mathcal{R}$ operator and the restriction-lifting composition $\mathcal{R}(\mathcal{L})$ are visualized in Fig.~\ref{fig:testOp_bif_TrM}, indicating that the composition $\mathcal{R}(\mathcal{L})$ using GHs for lifting is two orders of magnitude more accurate  than that of k-NN (mean relative error with GHs $<0.01\%$, whilst $1.63\%$ with k-NN).

%
%
%
\subsection{Step C. Diffusion Maps and Numerical Approximation Accuracy of the Restriction and Lifting operators for the Embedded Control}

In order to increase the numerical accuracy for the implementation of the embedded controller, we considered a local data set around the embedded steady-states that we
seek to stabilize, this time over the control parameter space.~The employment of the ``local'' DMs for the discovery of the appropriate set of coarse-grained variables that parameterize the manifold around the steady-states that we
seek to stabilize is shown in Fig.~\ref{fig:DMs_Cont_TrM}.~It is shown that $(\mathbf{y}_1, \mathbf{y}_2)$ are an appropriate set of DMs coordinated to parametrize the manifold, since they retrieve the geometry of employed the data set and a large spectral gap develops among $\lambda_2$ and $\lambda_3$.
\begin{figure}[!ht]
	\centering
	\includegraphics[width=0.94\textwidth]{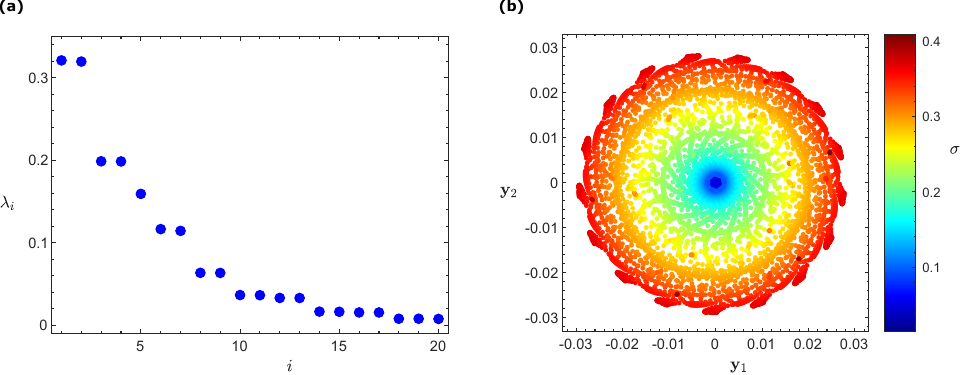} 
	\caption{~``Local'' DMs over the control parameter space: (a) the first 20 DMs eigenvalues and (b) the first two DMs coordinates $(\mathbf{y}_1, \mathbf{y}_2)$, colored by the standard deviation $\sigma$.~The results are obtained on a data set $\mathbf{X}\in \mathbb{R}^{10000\times 20}$ constructed via agent-based simulations at $v_0=1.0099$, thus varying the control parameter $h\in [2.39, 2.41]$.}
	\label{fig:DMs_Cont_TrM}
\end{figure}
Here, we report the numerical approximation accuracy of the restriction and lifting operators constructed on the basis of the DMs embeddings over the control parameter space.~Using the LOOCV procedure described in Section~\ref{subsub:LOOCV}, we assessed  the numerical approximation performance of the out-of-sample extension and the pre-image problems, this time by considering a randomly distributed subset of $5,000$ observations drawn from the new local data set $\mathbf{X}$.~The resulting numerical approximation errors in terms of $L_1$, $L_2$ and $L_{\infty}$ norms, are given in Table~\ref{tab:NormsCont}; 95\% confidence intervals (CI) are also provided.
\begin{table}[!h]
    \centering
    \caption{Numerical approximation accuracy in terms of $L_1$, $L_2$ and $L_{\infty}$ norms, of the restriction operator (obtained by coupling DMs with the Nystr\"{o}m method) (left column) and the lifting operator (obtained by coupling DMs with GHs and k-NN ) (middle and right columns) on the data set used for the construction of the wash-out embedded controller.}
    \begin{tabular}{l| c c | c c | c c}
    \toprule
    Operator   & \multicolumn{2}{c}{Restriction $\mathcal{R}$} & \multicolumn{2}{|c}{Lifting $\mathcal{L}$ with GHs} & \multicolumn{2}{|c}{Lifting $\mathcal{L}$ with k-NN} \\
    \midrule
    Error norm & mean $\times 10^{-5}$ & 95\% CI $\times 10^{-5}$ & mean $\times 10^{-2}$ & 95\% CI $\times 10^{-2}$ & mean $\times 10^{-2}$ & 95\% CI $\times 10^{-2}$ \\
    \midrule
    $\lVert \cdot \rVert_1$ & 3.385  & (3.277, 3.492) & 12.417 & (11.800,13.034) & 12.308 & (11.698,12.919) \\
    $\lVert \cdot \rVert_2$ & 2.659  & (2.575, 2.743) & 3.201  & (3.043,3.358) & 3.299 & (3.143,3.455)  \\
    $\lVert \cdot \rVert_{\infty}$ & 2.394  & (2.318, 2.470) & 1.312 & (1.248,1.376) & 1.463 & (1.399,1.526) \\
    \bottomrule
    \end{tabular}
    \label{tab:NormsCont}
    \vspace{-10pt}
\end{table}
For reconstructing the ambient space with GHs, we have set $\tilde{\epsilon}=7.23 \times 10^{-3}$ and $\delta = 5.58\times 10^{-8}$. This results to 500 eigenvectors of the second DMs embeddings, while for the reconstruction with the k-NN algorithm, we considered $k=8$ neighbors.~As shown in Table~\ref{tab:NormsCont}, a high numerical approximation accuracy is achieved for the restriction and lifting operators around the unstable equilibrium; lifting with GHs is slightly more accurate than lifting with k-NN.~Regarding the absolute and relative errors of the restriction $\mathcal{R}$ operator and the restriction-lifting composition $\mathcal{R}(\mathcal{L})$, as visualized in Fig.~\ref{fig:testOp_cont_TrM}, the composition $\mathcal{R}(\mathcal{L})$ using GHs for lifting is two orders of magnitude more accurate than that of k-NN.
\begin{figure}[!hb]
     \centering
     \includegraphics[width=\textwidth]{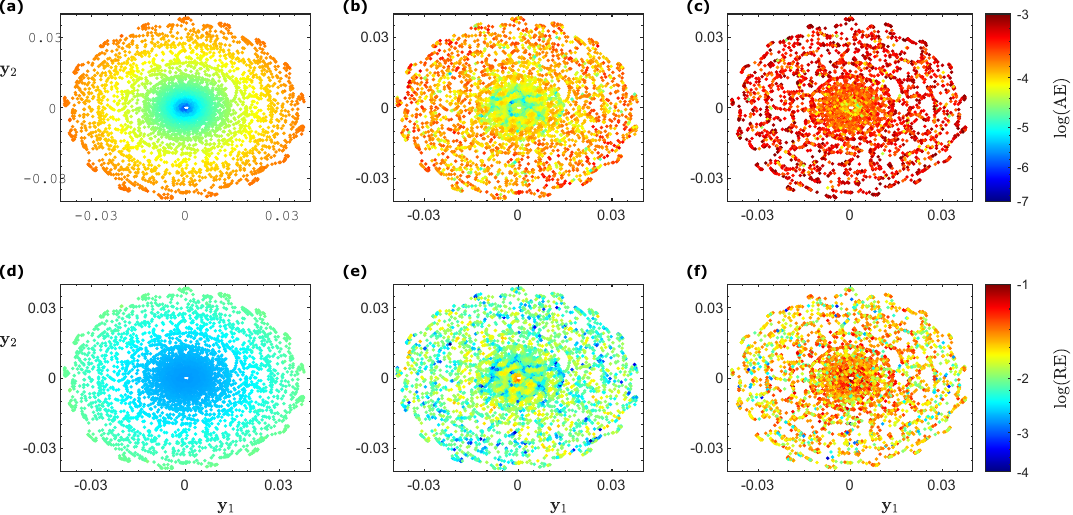}
    \caption{~Visualization of the numerical approximation accuracy of the restriction operator $\mathcal{R}$, and the restriction-lifting composition $\mathcal{R} \circ \mathcal{L}$, based on the data set of $\mathbf{X}\in \mathbb{R}^{10000 \times 20}$ points used to perform DMs over the control parameter space.~The numerical approximation accuracy was obtained through a LOOCV procedure in a subset of $5,000$ randomly distributed data points of $\mathbf{X}\in \mathbb{R}^{10000 \times 20}$.~Panels (a-c) and (d-f) depict the absolute and relative  $L_2$ norm errors, respectively.~Panels (a,d) correspond to the restriction operator (obained by coupling DMs with the Nystr\"{o}m method) and panels (b,e) and (c,f) correspond to the restriction-lifting composition $\mathcal{R} (\mathcal{L})$ using GHs and k-NN for lifting, respectively.}
    \label{fig:testOp_cont_TrM}
\end{figure}

\clearpage
\newpage
%
%
%
\section{The Stochastic Agent-based Model of a Financial Market with Mimesis}

\subsection{Step A. Numerical Approximation Accuracy of the Restriction and Lifting operators for the Coarse-grained Numerical Bifurcation Analysis}

Here, we report the numerical approximation accuracy of the restriction and lifting operators constructed on the basis of the DMs embeddings over the bifurcation parameter space using the LOOCV procedure described in Section~\ref{subsub:LOOCV}.~In Table~\ref{tab:Norms_InvM}, we report the mean numerical approximation errors of the two operators in terms of $L_1$, $L_2$ and $L_{\infty}$ norms; the 95\% CIs are also reported.~Note that the $L_1$-norm error of the lifting operator is of $\mathcal{O}(10)$, due to the high number $N=5,000$ of agents considered.
\begin{table}[!h]
    \centering
    \caption{Numerical approximation accuracy in terms of $L_1$, $L_2$ and $L_{\infty}$ norms of the restriction operator (obtained by coupling DMs with the Nystr\"{o}m method) (left column) and the lifting operator (obtained by coupling DMS with GHs and k-NN) (middle and right columns) that are used for the construction of the bifurcation diagram.}
    \begin{tabular}{l| c c | c c | c c}
    \toprule
    Operator   & \multicolumn{2}{c}{Restriction $\mathcal{R}$} & \multicolumn{2}{|c}{Lifting $\mathcal{L}$ with GHs} & \multicolumn{2}{|c}{Lifting $\mathcal{L}$ with k-NN} \\
    \midrule
    Error norm & mean $\times 10^{-6}$ & 95\% CI $\times 10^{-6}$ & mean $\times 10^{-1}$ & 95\% CI $\times 10^{-1}$ & mean $\times 10^{-1}$ & 95\% CI $\times 10^{-1}$ \\
    \midrule
    $\lVert \cdot \rVert_1$ & 1.210  & (1.103, 1.317) & 192.12 & (190.59, 190.37) & 163.73 & (162.43,165.02) \\
    $\lVert \cdot \rVert_2$ & 1.210  & (1.103, 1.317) & 3.806  & (3.790, 3.832) & 3.243 & (3.222, 3.265)  \\
    $\lVert \cdot \rVert_{\infty}$ & 1.210  & (1.103, 1.317) & 0.578 & (0.571, 0.586) & 0.492 & (0.485,0.498) \\
    \bottomrule
    \end{tabular}
    \label{tab:Norms_InvM}   
\end{table}
\begin{figure}[!hb]
     \centering
     \includegraphics[width=\textwidth]{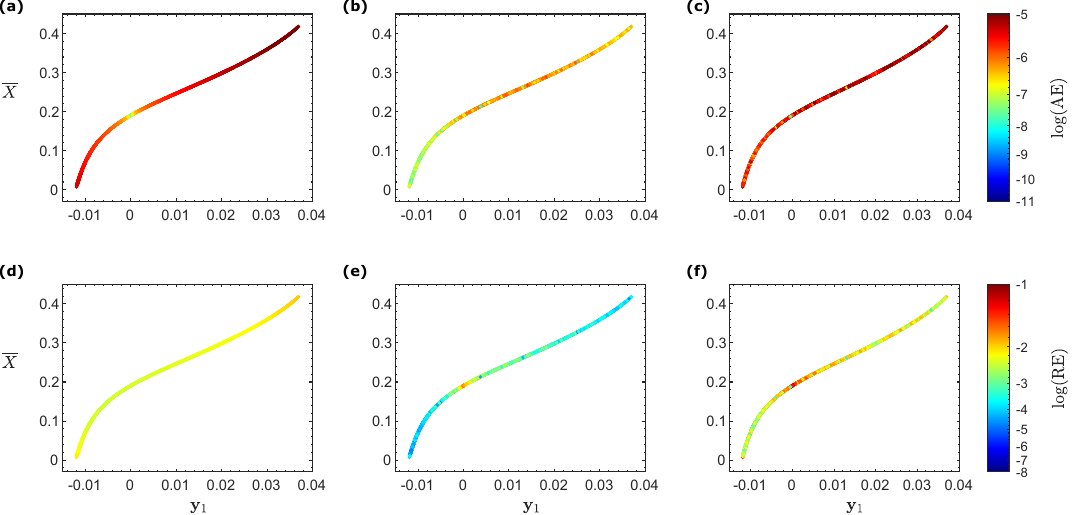}
     \caption{~Visualization of the numerical approximation accuracy of the restriction operator $\mathcal{R}$, and the restriction-lifting $\mathcal{R} \circ \mathcal{L}$ composition, based on the data set of $\mathbf{X}\in \mathbb{R}^{5000 \times 5000}$ points used to perform DMs over the bifurcation parameter space.~The numerical approximation accuracy was obtained through a LOOCV procedure.~For visualization purposes, we plot the first DMs coordinates $\mathbf{y}_1$ with respect to $\overline{X}$.~Panels (a-c) and (d-f) depict the absolute and relative $L_2$ norm errors, respectively.~Panels (a,d) correspond to the restriction operator (obtained by coupling DMs with the Nystr\"{o}m method) and panels (b,e) and (c,f) correspond to the restriction-lifting composition $\mathcal{R} (\mathcal{L})$ using GHs and k-NN for lifting, respectively.}
        \label{fig:testOp_TiM}
\end{figure}
For implementing the lifting operator with GHs, we have set $\tilde{\epsilon}=3.67 \times 10^{-4}$ and $\delta= 3.97 \times 10^{-11}$. This results to 400 eigenvectors for the second DMs embeddings.~For implementing the lifting operator with k-NN, we considered $k=3$ neighbors, since no significant accuracy improvement was reported with more neighbors.~As shown in Table~\ref{tab:Norms_InvM}, high numerical approximation accuracy is reported for the restriction operator (obtained by coupling DMs with the Nystr\"{o}m method) and the lifting operator (obtained by coupling DMs with both GHs and k-NN algorithms).~A visualization of the absolute and relative errors of the restriction $\mathcal{R}$ operator and restriction-lifting composition $\mathcal{R}(\mathcal{L})$ is provided in Fig.~\ref{fig:testOp_TiM}, further indicating that the composition $\mathcal{R}(\mathcal{L})$ using GHs for lifting is by two orders of magnitude more accurate than that of k-NN  (mean relative error with GHs 0.0039\%, whilst 0.082\% with k-NN).

\subsection{Step C. Diffusion Maps and Numerical Approximation Accuracy of the Restriction and Lifting operators for Embedded Control}

For increasing the numerical accuracy of the embedded controller, we considered a local data set around the embedded steady-state that we seek to stabilize, this time over the control parameter space.~The employment of the ``local'' DMs for the discovery of the appropriate set of coarse-grained variables that parameterize the manifold around the steady-states that we
seek to stabilize is shown in Fig.~\ref{fig:DMs_Cont_TiM}.~It is shown that  $\mathbf{y}_1$ is an appropriate coordinate for parametrizing the manifold, since a large spectral gap develops among $\lambda_1$ and $\lambda_2$ and $\mathbf{y}_1$ is in one-to-one relation with the first-order moment $\overline{X}=E[X_i]$.
\begin{figure}[!ht]
	\centering
	\includegraphics[width=\textwidth]{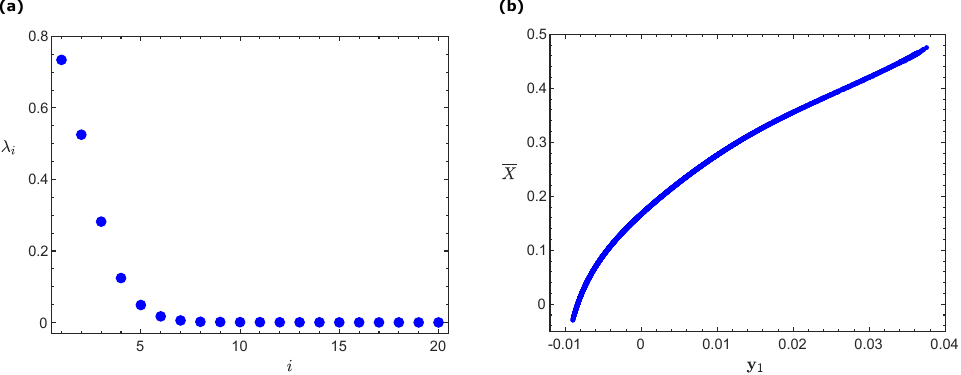}
	\caption{~``Local'' DMs over the control parameter space: (a) the first 20 DMs eigenvalues, and (b) the first DM embedding $\mathbf{y}_1$ in comparison to mean $\overline{X}$.~The data set $\mathbf{X} \in \mathbb{R}^{8343 \times 5000}$ for the data-driven controller design includes trajectories of agent-based simulations with fixed $g=41$, thus varying the control parameter $\nu_{ex}^-\in[19,21]$.}
	\label{fig:DMs_Cont_TiM}
\end{figure}
In order to assess the numerical approximation accuracy of the restriction and lifting operators constructed on the basis of the DMs embeddings over the control parameter space, we used the LOOCV procedure described in Section~\ref{subsub:LOOCV} to a randomly distributed subset of $5,000$ observations drawn from the new local data set $\mathbf{X}$.~Table~\ref{tab:Norms_InvM_Cont} reports the numerical approximation errors of the restriction operator $\mathcal{R}$ (obtained coupling DMs with the Nystr\"{o}m method), and the lifting operator $\mathcal{L}$ (obtained by coupling DMs with GHs and k-NN algorithms).
\begin{table}[ht]
    \centering
    \caption{Numerical approximation accuracy in terms of $L_1$, $L_2$ and $L_{\infty}$ norms of the restriction operator (obtained by coupling DMS with the Nystr\"{o}m method) (left column) and the lifting operator (obtained by coupling DMs with GHs and k-NN) (middle and right columns) on the data set used for the design of the wash-out embedded controller.}
    \begin{tabular}{l| c c | c c | c c}
    \toprule
    Operator   & \multicolumn{2}{c}{Restriction $\mathcal{R}$} & \multicolumn{2}{|c}{Lifting $\mathcal{L}$ with GHs} & \multicolumn{2}{|c}{Lifting $\mathcal{L}$ with k-NN} \\
    \midrule
    Error norm & mean $\times 10^{-6}$ & 95\% CI $\times 10^{-6}$ & mean $\times 10^{-1}$ & 95\% CI $\times 10^{-1}$ & mean $\times 10^{-1}$ & 95\% CI $\times 10^{-1}$ \\
    \midrule
    $\lVert \cdot \rVert_1$ & 0.964  & (0.868, 1.060) & 200.97 & (199.25, 202.69) & 169.07 & (167.65, 170.48) \\
    $\lVert \cdot \rVert_2$ & 0.964  & (0.868, 1.060) & 3.958   & (3.929, 3.987) & 3.332 & (3.308, 3.356)  \\
    $\lVert \cdot \rVert_{\infty}$ & 0.964  & (0.868, 1.060) & 0.572  & (0.564, 0.579) & 0.476 & (0.469, 0.482) \\
    \bottomrule
    \end{tabular}
    \label{tab:Norms_InvM_Cont}   
\end{table} 
~The approximation errors are calculated in terms of $L_1$, $L_2$ and $L_{\infty}$ norms; the 95\% CIs are also reported.~Note that the realtively large $L_1$-norm error of the lifting operator with both GH and k-NN is due to the size ($N=5,000$) of the ambient space.~For reconstructing the ambient space with GHs, we have set $\tilde{\epsilon}=3.14 \times 10^{-4}$ and $\delta = 1.05\times 10^{-8}$, thus resulting to 380 eigenvectors of the second local DMs embeddings, while for the reconstruction with the k-NN algorithm, we considered $k=3$ neighbors.~As shown in Table~\ref{tab:Norms_InvM_Cont}, the restriction and lifting operators around the unstable equilibrium provide a good numerical accuracy.~In Fig.~\ref{fig:testOp_cont_TiM}, we also report the absolute and relative errors of the restriction $\mathcal{R}$ operator and the restriction-lifting composition $\mathcal{R}(\mathcal{L})$.~It is again shown that using GHs for  lifting is by almost two orders of magnitude more accurate than that of k-NN.
\begin{figure}[!hb]
     \centering
     \includegraphics[width=\textwidth]{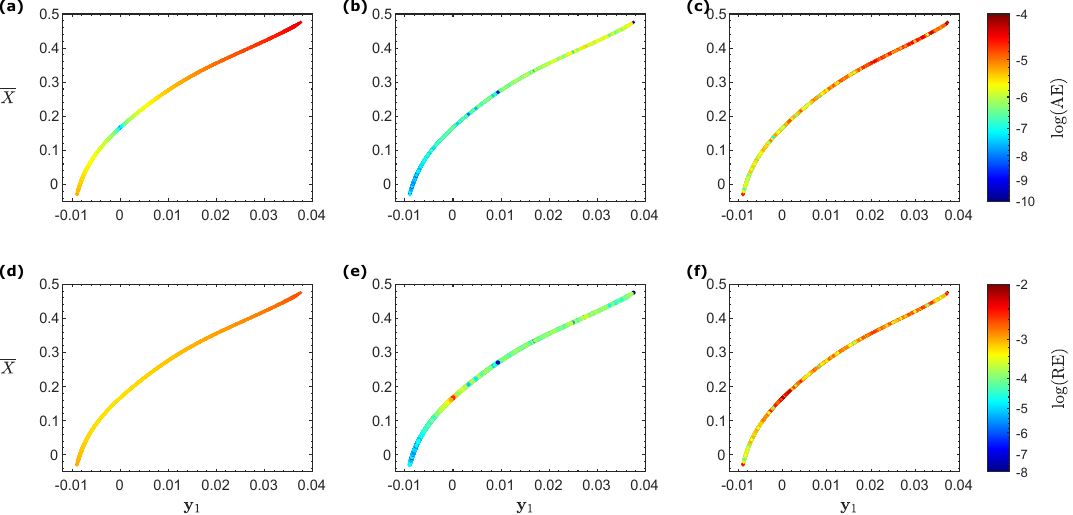}     
     \caption{Visualization of the numerical approximation accuracy of the restriction operator $\mathcal{R}$  and the restriction-lifting composition $\mathcal{R} \circ \mathcal{L}$ based on the data set of $\mathbf{X}\in \mathbb{R}^{8343 \times 5000}$ points used to perform DMs over the control parameter space.~The numerical approximation accuracy was obtained through a LOOCV procedure in a subset of $5,000$ randomly distributed data points of $\mathbf{X}\in \mathbb{R}^{8343 \times 5000}$.~For visualization purposes, we plot $\mathbf{y}_1$ with respect to $\overline{X}$.~Panels (a-c) and (d-f) depict the absolute and relative $L_2$ norm errors, respectively.~Panels (a,d) correspond to the restriction operator (obtained by coupling DMs with the Nystr\"{o}m method), and panels (b,e) and (c,f) correspond to the eestriction-lifting composition $\mathcal{R} (\mathcal{L})$, using GHs and k-NN for lifting, respectively.}
        \label{fig:testOp_cont_TiM}
\end{figure}

\end{document}